\documentclass[12pt,a4paper]{amsart}
\usepackage[T1]{fontenc}
\usepackage[english]{babel}
\usepackage{amsthm,amssymb,amsfonts,mathrsfs,amsmath,fancyhdr}
\usepackage{enumitem}
\usepackage{scalerel}
\usepackage{rotating}

\usepackage[maxbibnames=99, backend=biber,style=numeric, backref, backrefstyle=none, maxalphanames=4]{biblatex}
\usepackage[usenames,dvipsnames,usenames,dvipsnames,svgnames,table]{xcolor}
\definecolor{newblue}{RGB}{0,100,200}
\usepackage[nodisplayskipstretch]{setspace}
\usepackage[unicode,pdftex, bookmarks, linktocpage=true]{hyperref}
\hypersetup{hypertexnames=false, colorlinks=true, citecolor=blue, linkcolor=blue, urlcolor=magenta, pdfpagemode=UseNone , breaklinks=true}
\usepackage[babel]{csquotes}
\usepackage[left=2.8cm, right=2.8cm,top=2.6cm,bottom=2.6cm]{geometry}
\usepackage[all]{xy}

\usepackage{orcidlink}

\usepackage{dsfont}
\usepackage{bookmark}

\usepackage{tikz}
\usetikzlibrary{matrix,arrows,calc,decorations,patterns,decorations.markings}
\usepackage{tikz-cd}

\usepackage{bm}

\usepackage{etoolbox}
\patchcmd{\section}{\normalfont}{\normalfont\normalsize}{}{}

\usepackage{bbm}
\usepackage{tkz-euclide}
\tikzset{
  dotted/.style={pattern=dots,pattern color=#1},
  dotted/.default=black
}
\tikzset{
  fdotted/.style={pattern=crosshatch dots,pattern color=#1},
  fdotted/.default=black
}
\tikzset{
  scopedlines/.style={pattern=north east lines,pattern color=#1},
  scopedlines/.default=black
}
\tikzset{
  hrlines/.style={pattern=horizontal lines,pattern color=#1},
  hrlines/.default=black
}
\newcommand*{\DashedArrow}[1][]{\mathbin{\tikz [baseline=-0.25ex,-latex, dashed,#1] \draw [#1] (0pt,0.5ex) -- (1.3em,0.5ex);}}
\usepackage{microtype}

\usepackage{mathtools}

\theoremstyle{plain}
\newtheorem{lem}{Lemma}[section]

\theoremstyle{definition}
\newtheorem{defn}[lem]{Definition}

\theoremstyle{definition}
\newtheorem{rmk}[lem]{Remark}

\theoremstyle{plain}
\newtheorem{cor}[lem]{Corollary}

\theoremstyle{plain}
\newtheorem{prop}[lem]{Proposition}

\theoremstyle{plain}
\newtheorem{thm}[lem]{Theorem}

\theoremstyle{plain}
\newtheorem{conj}[lem]{Conjecture}

\theoremstyle{definition}
\newtheorem{ex}[lem]{Example}

\newenvironment{proof1}{\vspace{0.2cm}\paragraph{\bf\textit{Proof  of Theorem \ref{thmden}:}}}{\hfill$\blacksquare$ \medskip}

\newenvironment{proof2}{\vspace{0.2cm}\paragraph{\bf\textit{Proof  of Corollary \ref{cor1}:}}}{\hfill$\blacksquare$ \medskip}

\newenvironment{proof3}{\vspace{0.2cm}\paragraph{\bf\textit{Proof  of Corollary \ref{thm2}:}}}{\hfill$\blacksquare$ \medskip}

\newenvironment{proof4}{\vspace{0.2cm}\paragraph{\bf\textit{Proof  of Corollary \ref{corHoring}:}}}{\hfill$\blacksquare$ \medskip}

\newenvironment{proof5}{\vspace{0.2cm}\paragraph{\bf\textit{Proof  of Proposition \ref{thmden2}:}}}{\hfill$\blacksquare$ \medskip}

\newenvironment{alternativeproof}{\vspace{0.2cm}\paragraph{\bf\textit{Alternative proof of Theorem \ref{thmden}:}}}{\hfill$\blacksquare$ \medskip}

\newcommand{\EffS}{\overline{\mathrm{Eff}(S)}}
\newcommand{\EffX}{\overline{\mathrm{Eff}(X)}}

\newcommand{\MovX}{\overline{\mathrm{Mov}(X)}}
\newcommand{\Bir}{\mathrm{Bir}(X)}

\title[\scriptsize Pseudo-effective classes on IHS manifolds]{{Pseudo-effective classes on projective irreducible holomorphic symplectic manifolds}}
\author[Francesco Antonio Denisi]{ Francesco Antonio Denisi \orcidlink{0000-0002-1128-7890}}

\address{Université de Lorraine, Institut Elie Cartan de Lorraine, F-54506 Vandœuvre-lès-Nancy Cedex, France}
\curraddr{Université Paris Cité and Sorbonne Université, CNRS, IMJ-PRG, F-75013 Paris, France}
\email{denisi@imj-prg.fr}

\begin{document}
\begin{abstract}
We show that Kovács' result on the cone of curves of a K3 surface generalizes to any projective irreducible holomorphic symplectic manifold $X$. In particular, we show that if $\rho(X)\geq 3$, the pseudo-effective cone $\overline{\mathrm{Eff}(X)}$ is either circular or equal to $\overline{\sum_{E}\mathbf{R}^{\geq 0} [E]}$, where the sum runs over the prime exceptional divisors of $X$. The proof goes through hyperbolic geometry and the fact that (the image of) the Hodge monodromy group $\mathrm{Mon}^2_{\mathrm{Hdg}}(X)$ in $\text{O}^+(N^1(X))$ is of finite index. If $X$ belongs to one of the known deformation classes, carries a prime exceptional divisor $E$, and $\rho(X)\geq 3$, we explicitly construct an additional integral effective divisor, not numerically equivalent to $E$, with the same monodromy orbit as that of $E$. To conclude, we provide some consequences of the main result of the paper, for instance, we obtain the existence of uniruled divisors on certain primitive symplectic varieties.
\end{abstract}
\maketitle

\section{Introduction}
An irreducible holomorphic symplectic (IHS) manifold $X$ is a simply connected compact Kähler manifold such that $H^0(X,\Omega^2_X) \cong \mathbf{C} \sigma $, where $\sigma$ is a holomorphic symplectic form. They are the higher dimensional analogs of complex K3 surfaces, which are the IHS manifolds of dimension $2$, and one of the three building blocks of compact Kähler manifolds with vanishing first (real) Chern class. At present, only the K3$^{[n]}$ and $\mathrm{Kum}_n$ deformation classes, for any $n \in \mathbf{N}$ (see \cite{Beau}), and the OG10, OG6 deformation classes, discovered by O'Grady (see \cite{OGrady1} and \cite{OGrady2} respectively) are known. Thanks to the existence of a quadratic form $q_X$ on $H^2(X,\mathbf{C})$, known as the Beauville-Bogomolov-Fujiki (BBF) form, the structure of the pseudo-effective cone of a projective IHS manifold $X$ is similar to that of a smooth complex projective surface. Indeed, for a surface $S$ we have 
\[
\EffS=\overline{\mathscr{C}_S}+\sum_{C \in \mathrm{Neg}(S)}\mathbf{R}^{\geq 0} [C]
\]
(see \cite[Theorem 4.13]{Kollar}), where $\mathrm{Neg}(S)$ is the set of irreducible and reduced curves of negative self-intersection in $S$, and $\mathscr{C}_S$ is the positive cone of $S$, i.e. the connected component of the set $\{\alpha \in N^1(S)_{\mathbf{R}} \; | \; \alpha^2>0\}$ containing the ample cone $\mathrm{Amp}(S)$; while for a projective IHS manifold $X$ we have 
\[
\EffX= \overline{\mathscr{C}_X} \cap N^1(X)_{\mathbf{R}}+\sum_{E \in \mathrm{Neg}(X)}\mathbf{R}^{\geq 0}[E],
\] 
where $\mathrm{Neg}(X)$ is the set of prime exceptional divisors in $X$, i.e. the prime divisors with negative BBF square (see \cite[Corollary 3.5]{Den}), and $\mathscr{C}_X$ is the positive cone of $X$, i.e. the connected component of the set $\{\alpha \in H^{1,1}(X,\mathbf{R}) \; | \; q_X(\alpha)>0\}$ containing the K\"ahler cone $\mathscr{K}_X$. We are mainly interested in the algebraic part of $\mathscr{C}_X$ (i.e. $\mathscr{C}_X\cap N^1(X)_{\mathbf{R}}$). With some abuse of notation, from now on we will denote $\mathscr{C}_X\cap N^1(X)_{\mathbf{R}}$ by $\mathscr{C}_X$.
For a projective K3 surface (over an arbitrary algebraically closed field of any characteristic), Kovács obtained a refinement for the structure of the pseudo-effective cone. In particular, he proved the following.

\begin{thm}[\cite{Kov}, Theorem  1.1]\label{thmkov}
Let $S$ be a K3 surface of Picard number at least 3 over an algebraically closed field of arbitrary characteristic. Then either $\mathrm{Neg}(S)=\emptyset$ and $\overline{\mathrm{Eff}(S)}=\overline{\mathscr{C}_S}$ is circular, or \[
\overline{\mathrm{Eff}(S)}=\overline{\sum_{C \in \mathrm{Neg}(S)}\mathbf{R}^{\geq 0} [C]},
\] 
where $\mathrm{Neg}(S)$ is the set of smooth rational curves contained in $S$.
\end{thm}

Note that Kovács first proved the result above over $\mathbf{C}$ (cf. \cite[Theorem 1]{Kov1993}). See also  \cite[Theorem 2]{Kov1993} for a complete description of the pseudo-effective cone in this case.
 
Replacing the smooth rational curves with the prime exceptional divisors, it is natural, by all the above, to ask whether the result can be generalized to the case of an arbitrary projective IHS manifold. We can answer this question affirmatively. In particular, we obtained the following theorem, which is the main result of this paper.
 
\begin{thm}\label{thmden}
Let $X$ be a projective IHS manifold of Picard number greater than or equal to 3. Then either $\mathrm{Neg}(X)=\emptyset$ and $\overline{\mathrm{Eff}(X)}=\overline{\mathscr{C}_X}$ is circular, or \[
\overline{\mathrm{Eff}(X)}=\overline{\sum_{E \in \mathrm{Neg}(X)}\mathbf{R}^{\geq 0} [E]},\]
 where $\mathrm{Neg}(X)$ is the set of prime exceptional divisors in $X$.
\end{thm}

Note that the result above holds also in the singular setting (under some assumptions) and we refer the reader to Theorem \ref{thmden3} for a sketch of the proof in this case.
\vspace{0.2cm}

If one wants to adopt Kovács' strategy to prove Theorem \ref{thmden}, one quickly realizes that there is a problem: the class $\alpha$ constructed starting from a prime exceptional divisor $E$ and appearing in equation (\ref{alpha}) of Lemma \ref{tecnolemma} is not effective, in general. Indeed, while in the case of K3 surfaces the effectivity of $\alpha$ directly follows from the Riemann-Roch Theorem for surfaces, in the higher dimensional case the Riemann-Roch Theorem for IHS manifolds does not tell us anything about the effectivity of $\alpha$, mainly because we do not have vanishing theorems helping us. 
Our idea to bypass this problem is to use the monodromy group of $X$. More precisely, in our original approach, see Proposition \ref{thmden2} below and the of Section \ref{Section3}, we used the explicit description of $\mathrm{Mon}^2_{\mathrm{Hdg}}(X)$ for the known deformation classes. A stronger approach in two steps is to use
the fact that (the image of) $\mathrm{Mon}^2_{\mathrm{Hdg}}(X)$ in $\text{O}^+(\mathrm{Pic}(X))$ is of finite index to prove that
it acts on the boundary of the projectivized positive cone with dense orbits and to deduce from that Theorem \ref{thmden}. That is what we do in Section \ref{Section2}.
\vspace{0.2cm}

On the other hand, we use Kovács' approach to produce explicit effective integral divisors, either of BBF square $0$, or with the same monodromy orbit as that of a given prime exceptional divisor, but this construction (apparently) works only for the known deformation classes of IHS manifolds. 

\begin{prop}\label{thmden2}
Let $X$ be a projective IHS manifold belonging to one of the known deformation classes, with $\rho(X)\geq 2$, and carrying a primitive class $\beta$, such that $k\beta=[E]$, for some prime exceptional divisor $E$ and some positive integer $k$. Then, for any integral divisor class $\gamma$ lying in $\mathscr{C}_X$, there exists an effective class of the form $x\gamma-y \beta$ (with $x,y \in \mathbf{N}$) which is of BBF square 0, or with the same monodromy orbit as that of $\beta$.
\end{prop}

\begin{rmk}
The divisors in the proposition above are constructed as solutions of a certain Pell equation. This does not come as a surprise. Indeed, Pell's equations play an important role in describing cones of divisors of IHS manifolds (see for instance \cite{BM13}, \cite{BC22}), and we hope that our explicit computations can be helpful in similar problems. It could happen that, in the situation of Proposition \ref{thmden2}, both the mentioned types of divisors exist at once (e.g.\ \cite[Example 6.2]{DenOrtiz}). But, in general, this is not true (e.g.\ \cite[Example 6.1]{DenOrtiz}). Anyhow, by making use of the solutions of the cited Pell equation, we can construct explicit effective divisors with the same monodromy orbit as that of the prime exceptional divisor $E$, if the Picard number of $X$ is at least 3 (see Remark \ref{effectivedivisors}).
\end{rmk}

Theorem \ref{thmden} describes the pseudo-effective cone of IHS manifolds of Picard number at least $3$. In the case of Picard number 2, we have the following result of Oguiso.

\begin{thm}[\cite{Ogu}, item (2) of Theorem 1.3]\label{thmogu}
Let $X$ be a projective IHS manifold of Picard number 2. Then either both the boundary rays of the movable cone $\MovX$ are rational and $\Bir$ is a finite group or both the boundary rays of $\MovX$ are irrational and $\Bir$ is an infinite group.
\end{thm}

From Theorem \ref{thmogu} one can easily deduce the following description of the pseudo-effective cone for projective IHS manifolds of Picard number 2.

\begin{cor}\label{cor1}
Let $X$ be a projective IHS manifold with $\rho(X)=2$. Then either the two boundary rays of $\EffX$ are rational and $\Bir$ is a finite group or the two boundary rays of $\EffX$ are irrational and $\Bir$ is an infinite group.
\end{cor}

We recall that the reflection associated with a prime exceptional divisor $E$ is the isometry
\[
R_E \colon H^2(X,\mathbf{Z}) \to H^2(X,\mathbf{Z}), \; \alpha \mapsto \alpha-\frac{2q_X([E],\alpha)}{q_X([E])}[E].
\]

A priori, it is not obvious that the reflection $R_E$ is integral. This has been proven by Markman (cf. \cite[Corollary 3.6, item (1)]{Mark2}). 

\begin{defn}\label{defwex}
We define $W_{\mathrm{Exc}}$ as the group generated by the reflections $R_E$, where $E$ is any stably exceptional divisor of $X$.
\end{defn}

Putting all together one obtains the following result, which is a higher dimensional analog of a result due to Pjatecki\u{i}-\v{S}apiro and \v{S}afarevi\v{c}, and Sterk (cf. \cite[Corollary 4.7]{Huy2}).

\begin{cor}\label{thm2}
Let $X$ be a projective IHS manifold belonging to one of the known deformation classes. Consider the following statements.
\begin{enumerate}
\item $\mathrm{Eff}(X)$ is rational polyhedral.
\item $\mathrm{Bir}(X)$ is finite.
\item $\mathrm{O}(N^1(X))/W_{\mathrm{Exc}}$ is finite.
\item $X$ carries finitely many prime exceptional divisors.
\end{enumerate}
Then the statements \textit{(1),(2),(3)} are equivalent and imply \textit{(4)}. Furthermore, if $X$ is of Picard number at least $3$ and carries a prime exceptional divisor, all the statements above are equivalent.
\end{cor}
We believe that the experts are aware of the equivalence between \textit{(1)}, \textit{(2)}, and \textit{(3)} in Corollary \ref{thm2}. The result is interesting on its own, and we include it for the sake of completeness.
\vspace{0.2cm}

\begin{defn}\label{defpsv}
A primitive symplectic variety is a normal compact Kähler variety $Y$ such that $h^1(Y,\mathscr{O}_Y)=0$ and $H^0\left(Y,\Omega^{[2]}_Y\right)$ is generated by a holomorphic symplectic form $\sigma$ such that $Y$ has symplectic singularities.
\end{defn}

We refer the reader to \cite{BL22} for the general theory of primitive symplectic varieties.
\vspace{0.2cm}

Using the convex geometry of the pseudo-effective cone of projective IHS manifolds we deduce the existence of (rigid) uniruled divisors on certain primitive symplectic varieties.

\begin{cor}\label{corHoring}
Let $Y$ be a singular $\mathbf{Q}$-factorial projective primitive symplectic variety admitting a resolution $f\colon X\to Y$, with $X$ a projective IHS manifold of Picard number $\rho(X)\geq 3$. Then $Y$ carries a prime exceptional (hence uniruled) divisor.
\end{cor}

The uniruledness of prime exceptional divisors on $\mathbf{Q}$-factorial primitive symplectic varieties has been proven in \cite[Theorem 1.2]{LMP23}. See also Corollary \ref{corsingular} for a more general version of Corollary \ref{corHoring}.

\subsection*{Organization of the paper}

\begin{enumerate}
\item In Section \ref{Section1} we collect all the notions and results that will be needed and respectively used to achieve our goals.

\item In Section \ref{Section2} we prove Theorem \ref{thmden} and present some consequences of it.

\item In Section \ref{Section3} we prove Proposition \ref{thmden2} and show how it can be used to prove Theorem \ref{thmden} without the use of hyperbolic geometry (but only for the known deformation classes).

\end{enumerate}

 \section*{Acknowledgements}
This work is part of my PhD thesis. First of all, I would like to thank my doctoral advisors Gianluca Pacienza and Giovanni Mongardi, for their suggestions and hints. I am grateful to thank F.\ Bastianelli,  A.\ F.\ Lopez, and F.\ Viviani, for pointing out to me Example \ref{ex1}. 
I would also like to thank A.\ Höring for pointing out Corollary \ref{corHoring},  
and S.\ Francaviglia, S.\ Riolo, D.\ Zheng-Xu for useful discussions in hyperbolic geometry.
Finally, I would like to thank the anonymous referee from the Annals of the Institute Fourier for carefully revising this paper. Her/His suggestions considerably improved the manuscript.

\section{Preliminaries}\label{Section1}
In this section, we collect the main definitions, tools, and results needed for the rest of the paper.
\subsection{Generalities on IHS manifolds}

For a general introduction to IHS manifolds, we refer the reader to \cite{Joyce}.
Let $X$ be an IHS manifold. Thanks to the work \cite{Beau} of Beauville, there exists a quadratic form on $H^2(X,\mathbf{C})$ generalizing the intersection form on a surface. In particular, choosing the symplectic form $\sigma$ in such a way that  $\int_X(\sigma\overline{\sigma})^n=1$, one can define
\[
q_X(\alpha):=\frac{n}{2}\int_X (\sigma\overline{\sigma})^{n-1}\alpha^2+(1-n)\left(\int_X\sigma^n\overline{\sigma}^{n-1}\alpha\right)\cdot \left(\int_X\sigma^{n-1}\overline{\sigma}^n\alpha\right),
\]
for any $\alpha \in H^2(X,\mathbf{C})$. The quadratic form $q_X$ is non-degenerate and is known as the Beauville-Bogomolov-Fujiki form (BBF form in what follows). Up to a rescaling $q_X$ is integral and primitive on $H^2(X,\mathbf{Z})$. Also, $q_X$ is invariant by deformations. Naturally associated with $q_X$ is a bilinear form, which we denote by $q_X(-,-)$. If $\alpha$ is any element of $H^2(X,\mathbf{C})$, we denote $q_X(\alpha,\alpha)$ by $q_X(\alpha)$.
A prime divisor in $X$ will be a reduced and irreducible hypersurface. We say that a prime divisor $E$ of $X$ is exceptional if $q_X(E)<0$.

Let $f \colon X \DashedArrow[->,densely dashed] X'$ be a bimeromorphic map, where $X'$ is another IHS manifold. Recall that $f$ restricts to an isomorphism $f \colon U \to U'$, where $\mathrm{codim}_{X}(X \setminus U)$, $\mathrm{codim}_{X'}(X' \setminus U')\geq 2$ and $X\setminus U$, $X' \setminus U'$ are analytic subsets of $X$ and $X'$ respectively (see for example \cite[paragraph 4.4]{Huy1}). Using the long exact sequence in cohomology with compact support and Poincaré duality, one has the usual chain of isomorphisms
\[
H^2(X, \mathbf{R}) \cong H^2(U, \mathbf{R}) \cong H^2(U', \mathbf{R}) \cong H^2(X',\mathbf{R}),
\]
and the composition is an isometry with respect to $q_X$ and $q_{X'}$ (see for example \cite[Proposition I.6.2]{OGrady} for a proof), hence its restriction to $H^{1,1}(X, \mathbf{R})$ induces an isometry $H^{1,1}(X, \mathbf{R}) \cong H^{1,1}(X', \mathbf{R})$, which we will denote by $f_{*}$ (push-forward). We will denote the inverse of $f_{*}$ by $f^{*}$ (pull-back).

\subsection{Some lattice theory}
A lattice is a couple $(L,b)$, where $L$ is a finitely generated, free abelian group, and $b$ is a non-degenerate, integral valued, symmetric bilinear form $b \colon L \times L \to \mathbf{Z}$. If no confusion arises, we will simply write $L$. We say that  $L$ is even if $b(x,x)$ is even for any $x \in L$. The signature $\mathrm{sign}(b)$ of $b$ is the signature of the natural extension $b_{\mathbf{R}}$ of $b$ to $L \otimes_{\mathbf{Z}} \mathbf{R}$.

The divisibility $\mathrm{div}_L(x)$ of an element $x \in L$ is defined as the positive generator of the ideal $b(x,L)$ in $\mathbf{Z}$. This means that $\mathrm{div}_L(x)$ is exactly the least (positive) integer that can be obtained by multiplying $x$ by the elements of $L$. When no confusion arises, we will write $\mathrm{div}(x)$ instead of $\mathrm{div}_L(x)$.

As $b$ is non-degenerate, one has an injective group homomorphism $L \hookrightarrow L^{\vee}:=\mathrm{Hom}_{\mathbf{Z}}(L,\mathbf{Z})$, by sending an element $x$ of $L$ to the element $b(x,-)$ of $L^{\vee}$. We will denote $b(x,-)$ by $x$ when no confusion arises. Note that we have the following identification
\[
L^{\vee}= \{x\in L\otimes_{\mathbf{Z}} \mathbf{Q} \; | \; b(x,y) \in \mathbf{Z}, \text{ for any } y \in L\}\subset L \otimes_{\mathbf{Z}} \mathbf{Q}.
\]

The discriminant group of $L$ is the finite group $A_L:=L^{\vee}/L$. If $A_L$ is trivial we say that $L$ is unimodular. If $f$ is an element of $L^{\vee}$ we will denote its class in $A_L$ by $[f]$. We say that an element $x \in L$ is primitive if we cannot write $x=ax'$, where $a\neq 1$, $a \in \mathbf{Z}^{>0}$, and $x' \in L$. 

An isometry of $L$ is an automorphism of it (as an abelian group) preserving $b$.  The group of isometries of $L$ is denoted by $\text{O}(L)$. Suppose now that $L$ has signature $\mathrm{sign}(b)=(3,b_2-3)$, where $b_2 \in \mathbf{N}$, and $b_2>3$. Define the cone
\[
C_L:=\{v \in L\otimes_{\mathbf{Z}} \mathbf{R} \;|\; b(v)>0\}.
\] 
In \cite[Lemma 4.1]{Mark} it has been proven that $C_L$ has the homotopy type of $S^2$, hence $H^2(C_L,\mathbf{Z})\cong \mathbf{Z}$. Any isometry in $\text{O}(L)$ induces a homeomorphism of $C_L$, which in turn induces an automorphism of $H^2(C_L,\mathbf{Z})\cong \mathbf{Z}$. This automorphism can act as $1$ or $-1$. We define $\text{O}^+(L)$ as the subgroup of the isometries of $L$ acting trivially on $H^2(C_L,\mathbf{Z})$. The group $\widetilde{\text{SO}}^+(L)$ is the subgroup of the isometries of $L$ of determinant 1, acting trivially on both $A_L$ and $H^2(C_L,\mathbf{Z})\cong \mathbf{Z}$.

\begin{rmk}\label{rmklattice} Given a primitive element $x\in L$, we can consider the element $x/\mathrm{div}(x)$ of $L^{\vee}$, which in turn gives the element $[x/\mathrm{div}(x)]$ of $A_L$. Note that $\mathrm{ord}([x/\mathrm{div}(x)])=\mathrm{div}(x)$, hence $\mathrm{div}(x)$ divides $|A_L|$. Indeed, suppose that the order is $k \leq \mathrm{div}(x)$, then $kt=\mathrm{div}(x)$ for some positive integer $t$. This implies that there exists an element $y \in L$ such that $x=ty$, and the primitivity of $x$ forces $t$ to be 1, and so $k=\mathrm{div}(x)$. Also, note that if $M$ is the maximum among the orders of the elements of $A_L$, then $\mathrm{div}(x)\leq M$. Indeed, the element $x/\mathrm{div}(x)$ has order $\mathrm{div}(x)$, thus $\mathrm{div}(x) \leq M$.
\end{rmk}

 The following is a result of Eichler, which will be very useful in this article. 

\begin{lem}[Eichler's criterion]\label{Eichlem}
Let $L'$ be an even lattice and $L=U^2\oplus L'$. Let $v,w \in L$ be two primitive elements such that the following hold:
\begin{itemize}
\item $b(v)=b(w)$.
\item $[v/\mathrm{div}(v)]=[w/\mathrm{div}(w)]$ in $A_L$.
\end{itemize}
Then there exists an isometry $\iota \in \widetilde{\mathrm{SO}}^+(L)$, such that $\iota(v)=w$.
\end{lem}

The version of Eichler's criterion given above has been taken from \cite[Lemma 2.6]{Mon1}.

Given an IHS manifold $X$, one has the lattice $(H^2(X,\mathbf{Z}),q_X)$, where (with some abuse of notation) by $q_X$ we mean the integral valued, symmetric bilinear form induced by the restriction of the BBF form to $H^2(X,\mathbf{Z})$. Note that, for any of the known IHS manifolds, the lattice $H^2(X,\mathbf{Z})$ is even and contains two copies of the hyperbolic lattice $U$ so that we can use Eichler's criterion. The signature of $q_X$ on $H^2(X,\mathbf{Z})$ is $(3, b_2(X)-3)$, where $b_2(X)$ is the second Betti number of $X$. We will denote the discriminant group $A_{H^2(X,\mathbf{Z})}$ by $A_X$. Also $(\mathrm{Pic}(X),q_X)$ is a lattice, and in this case, the signature of $q_X$ is $(1,\rho(X)-1)$, where $\rho(X)$ is the Picard number of $X$. We refer the reader to \cite[Corollary 23.11]{Joyce} for proof of these facts.

\subsection{Monodromy operators}
Let $\pi \colon \mathcal{X}\to S$ be a smooth and proper family of IHS manifolds over a connected analytic (possibly singular, reducible) base $S$, whose central fiber is a fixed IHS manifold $X$. Let $R^k\pi_{*}\mathbf{Z}$ be the $k$-th higher direct image of $\mathbf{Z}$ . The space $S$ is locally contractible, and the family $\pi$ is topologically locally trivial (see for example \cite[Theorem 14.5]{Joyce}). The sheaf $R^k\pi_{*}\mathbf{Z}$ is the sheafification of the presheaf
\[
U \mapsto H^k(\pi^{-1}(U),\mathbf{Z}), \; \text{ for any open subset } U\subset S,
\]   
which is a constant presheaf, by the local contractibility of $S$ and the local triviality of $\pi$. Then $R^k\pi_{*}\mathbf{Z}$ is a locally constant sheaf, i.e. a local system, for every $k \in \mathbf{N}$. Now, let $\gamma \colon [0,1] \to S$ be a continuous path. Then $\gamma^{-1}\left(R^k\pi_{*}\mathbf{Z}\right)$ is a constant sheaf for every $k$.

\begin{defn}
\begin{enumerate}
\item Set $\pi^{-1}(\gamma(0))=X$ and $\pi^{-1}(\gamma(1))=X'$. The parallel transport operator $T^k_{\gamma, \pi}$ associated with the path $\gamma$ and $\pi$ is the isomorphism $T^k_{\gamma,\pi}\colon H^k(X,\mathbf{Z})\to H^k(X',\mathbf{Z})$ between the stalks at 0 and 1 of the sheaf $R^k\pi_{*}\mathbf{Z}$, induced by the trivialization of $\gamma^{-1}\left(R^k\pi_{*}\mathbf{Z}\right)$. The isomorphism $T^k_{\gamma,\pi}$ is well defined on the fixed endpoints homotopy class of $\gamma$.
\item If $\gamma$ is a loop, we obtain an automorphism $T^k_{\gamma,\pi}\colon H^k(X,\mathbf{Z})\to H^k(X,\mathbf{Z})$. In this case $T^k_{\gamma,\pi}$ is called a monodromy operator.
\item The $k$-th group of monodromy operators on $H^k(X,\mathbf{Z})$ induced by $\pi$ is
\[
\mathrm{Mon}^k(X)_{\mathbf{\pi}}:=\{T^k_{\gamma,\pi} \; | \; \gamma(0)=\gamma(1) \}
\]
\end{enumerate}
\end{defn}
With the definition above we can define the monodromy groups of an IHS manifold $X$.
\begin{defn}
The $k$-th monodromy group $\mathrm{Mon}^k(X)$ of an IHS manifold $X$ is defined as the subgroup of $\mathrm{Aut}_{\mathbf{Z}}(H^k(X,\mathbf{Z}))$ generated by the subgroups of the type  $\mathrm{Mon}^k(X)_{\mathbf{\pi}}$, where $\pi \colon \mathcal{X} \to S$ is a smooth and proper family of IHS manifolds over a connected analytic base. The elements of $\mathrm{Mon}^k(X)$ are also called monodromy operators.
\end{defn}

We are interested in the group $\mathrm{Mon}^2(X)$, of which we will need the characterization for some of the known IHS manifolds. Notice that $\mathrm{Mon}^2(X) \subseteq \text{O}^+(H^2(X,\mathbf{Z}))$. We will mostly use three subgroups of $\mathrm{Mon}^2(X)$. The first is $\mathrm{Mon}^2_{\mathrm{Bir}}(X)$, defined as the elements of $\mathrm{Mon}^2(X)$ induced by the birational self maps of $X$. Indeed, any birational self-map induces a monodromy operator on $H^2(X,\mathbf{Z})$, by a result of Huybrechts (cf. \cite[Corollary 2.7]{Huy3}). The second is $\mathrm{Mon}_{\mathrm{Hdg}}^2(X)$, namely the subgroup of the elements of $\mathrm{Mon}^2(X)$ which are also Hodge isometries. The last one is $W_{\mathrm{Exc}}$ (see Definition \ref{defwex}). Indeed, in \cite[Corollary 3.6, item (1)]{Mark2} Markman proves that any reflection $R_E$ (associated with a prime exceptional divisor $E$) is an element of $\mathrm{Mon}_\mathrm{Hdg}^2(X)$.

\vspace{0.2cm}

Let $\mathrm{Def}(X)$ be the Kuranishi deformation space of any IHS manifold $X$ and $\mathcal{X}\to \mathrm{Def}(X)$ the universal family. Also, let $0$ be a distinguished point of $\mathrm{Def}(X)$ such that $\mathcal{X}_0 \cong X$. Set $\Lambda=H^2(X,\mathbf{Z})$. By the local Torelli Theorem, $\mathrm{Def}(X)$ embeds holomorphically into the period domain
\[
\Omega_{\Lambda}:=\{p \in \mathbf{P}(\Lambda \otimes_{\mathbf{Z}} \mathbf{C}) \; | \; q_X(p)=0 \text{ and } q_X(p+\overline{p})>0\}
\]
as an open (analytic) subset, via the local period map
\[
\mathcal{P}\colon \mathrm{Def}(X)\to \Omega_{\Lambda}, \; t \mapsto \left[H^{2,0}(\mathcal{X}_t)\right].
\]
Now, let $L$ be a holomorphic line bundle on $X$. The Kuranishi deformation space of the couple $(X,L)$ is defined as $\mathrm{Def}(X,L):= \mathcal{P}^{-1}(c_1(L)^{\perp})$, where $c_1(L)^{\perp}$ is a hyperplane in $\mathbf{P}(\Lambda \otimes_{\mathbf{Z}} \mathbf{C})$. The space $\mathrm{Def}(X,L)$ is the part of $\mathrm{Def}(X)$ where $c_1(L)$ stays algebraic. Up to shrinking $\mathrm{Def}(X)$ around $0$, we can assume that $\mathrm{Def}(X)$ and $\mathrm{Def}(X,L)$ are contractible. Let $s$ be the flat section (with respect to the Gauss-Manin connection) of $R^2\pi_{*}\mathbf{Z}$ through $c_1(L)$ and $s_t \in H^{1,1}(\mathcal{X}_t,\mathbf{Z})$ its value at $t \in \mathrm{Def}(X,L)$. 

The following are the line bundles we will mostly be interested in.

\begin{defn}\label{stablyexc}
A line bundle $L$ on an IHS manifold $X$ is stably exceptional if there exists a closed analytic subset $Z \subset \mathrm{Def}(X,L)$, of positive codimension, such that the linear system $|L_t|$ consists of a prime exceptional divisor for every $t \in \mathrm{Def}(X,L) \setminus Z$. 
\end{defn}

Prime exceptional divisors are stably exceptional, by \cite[Proposition 5.2]{Mark4}.
By making use of parallel transport operators, we can define "stably exceptional classes".

\begin{defn}\label{stablyexc1}
Let $X$ be a projective IHS manifold. A primitive, integral divisor class $\alpha \in N^1(X)$ (the Néron-Severi group of $X$) is stably exceptional if $q_X(\alpha,A)>0$, for some ample divisor $A$, and there exist a projective IHS manifold $X'$ and a parallel transport operator 
\[
f \colon H^2(X,\mathbf{Z})\to H^2(X',\mathbf{Z}),
\] 
such that $kf(\alpha)$ is represented by a prime exceptional divisor on $X'$, for some integer $k$.
\end{defn}

For example, a line bundle of self-intersection $-2$, intersecting positively any ample line bundle on a (smooth) projective K3 surface is stably exceptional, and its class in the Néron-Severi space is stably exceptional. Note that any stably exceptional class (line bundle) is effective, by the semi-continuity Theorem.
\vspace{0.2cm}

\subsection{Cones and positivity notions}
Throughout the rest of this section (if not otherwise stated) $X$ will be a projective IHS manifold, and $Y$ a smooth complex projective variety.
 \begin{defn}
 An integral divisor $D$ on $Y$ is big if there exists a positive integer $m$ such that $mD \equiv A+N$, where $A$ is integral and ample, $N$ is integral and effective, and $\equiv$ is the numerical equivalence relation. When $D$ is a $\mathbf{Q}$-divisor, we say that it is big if there exists a positive integer $m$ such that $mD$ is integral and big. If $D$ is an $\mathbf{R}$-divisor, we say that it is big if $D=\sum_ia_iD_i$, where $a_i$ is a positive real number and $D_i$ is integral and big, for any $i$.
 \end{defn}

Note that bigness behaves well with respect to $\equiv$, namely, if $D\equiv D'$, $D$ is big if and only if $D'$ is. The big cone $\mathrm{Big}(Y)$ of $Y$ is the convex cone in $N^1(Y)_{\mathbf{R}}$ of all big divisor classes.
\vspace{0.2cm}

Recall that the pseudo-effective cone $\overline{\mathrm{Eff}(Y)}$ of $Y$ is the closure of the big cone $\mathrm{Big}(Y)$ in the Néron-Severi space $N^1(Y)_{\mathbf{R}}$, which is defined as $N^1(Y)_{\mathbf{R}}:=N^1(X) \otimes \mathbf{R}$, where $N^1(Y)$ is the Néron-Severi group of $Y$. We notice that the big cone is the interior of $\overline{\mathrm{Eff}(Y)}$. For an excellent account of big divisors and cones in the Néron-Severi space on arbitrary complex projective varieties, we refer the reader to \cite{Laz}.
\vspace{0.2cm}

Let $D$ be any $\mathbf{R}$-divisor on $Y$. The \textit{augmented base locus} of $D$ is
\[
\mathbf{B}_{+}(D):=\bigcap_{D=A+E}\mathrm{Supp}(E),
\]
where the intersection is taken over all the decompositions of the form $D=A+E$, where $A$ and $E$ are $\mathbf{R}$-divisors such that $A$ is ample and $E$ is effective.

\begin{defn}
The birational K\"ahler cone $\mathscr{BK}_X$ is defined as
\[
\mathscr{BK}_X=\bigcup_{f} f^{*}\mathscr{K}_{X'},
\]
where $f$ varies among all the bimeromorphic maps $f \colon X \DashedArrow[->,densely dashed    ] X'$, where $X'$ is another IHS manifold (which must be projective).
\end{defn}

We now introduce another cone in the Néron-Severi space.

\begin{defn}
\begin{itemize}
\item  A movable line bundle on $Y$ is a line bundle $L$ such that the linear series $|mL|$ has no divisorial components in its base locus, for $m$ sufficiently large and divisible. A divisor $D$ is movable if $\mathscr{O}_Y(D)$ is.
\item The movable cone is the cone $\mathrm{Mov}(Y)$ in $N^1(Y)_{\mathbf{R}}$ generated by the classes of movable divisors. 
\end{itemize} 
\end{defn}

In the case of projective IHS manifolds, an important relation among the last two cones we have introduced is 
\begin{equation}
\MovX=\overline{\mathscr{BK}_X} \cap N^1(X)_{\mathbf{R}}
\end{equation}
(cf. \cite[Theorem 7]{Hassett}), where the closure of the birational K\"{a}hler cone is taken in $H^{1,1}(X,\mathbf{R})$.

We now recall a couple of definitions concerned with the geometry of convex cones.
\begin{defn}\label{circpart}
Let $K\subset \mathbf{R}^k$ be a closed convex cone with a non-empty interior. 
\begin{itemize}
\item We say that $K$ is locally (rational) polyhedral at $v \in \partial K$ if $v$ has an open neighborhood $U = U(v)$, such that $K \cap U$ is defined in $U$ by a finite number of (rational) linear inequalities.
\item Let $U$ be an open  subset of $\partial K$. We say that $U$ is a circular part of $K$ if there is no element in $U$ at which $K$ is locally polyhedral. 
\end{itemize}
\end{defn}

\begin{rmk}
Note that in the papers \cite{Kov1993}, \cite{Kov} Kovács adopted another definition of "circular part" for a convex cone, to prove his Theorem \ref{thmkov}. In particular, let $K$ be as in Definition \ref{circpart}. We say that $K$ is locally finitely generated at $v\in \partial K$ if there exists a finite set of points $S=\{v_1,\dots,v_k\}\subset K$ and a closed subcone $C \subset K$ such that $v\not\in C$ and $K$ is generated by $S$ and $C$. For Kovács, an open subset $U \subset \partial K$ is a circular part of $K$ if there is no point of $U$ at which $K$ is locally finitely generated. One can show that if $K$ is locally polyhedral at a point, then $K$ is locally finitely generated at that point. But the vice versa does not hold, in general (see Example \ref{ex1}). 
However, in Theorem \ref{thmden}, a circular part of the pseudo-effective cone (if there is any) with respect to Definition \ref{circpart} is circular with respect to Kovács' definition too, because such a circular part would be contained in $\partial \overline{\mathscr{C}_X}$ (see "Proof of Theorem 1.2"). Hence, no problem arises in using our definition. Note that Definition \ref{circpart} has also been adopted by Huybrechts in \cite[Chapter 8, Section 3]{Huy2}, where he provided a slightly different proof of Theorem \ref{thmkov}.
\end{rmk}

The following is an example of a locally finitely generated cone that is not locally polyhedral (at a point).

\begin{ex}\label{ex1}
Let $C$ be the "ice cream" in Figure 1 and $C \subset \mathbf{R}^3 \subset \mathbf{R}^4$ an embedding of $C$ in $\mathbf{R}^4$, such that the vertex $V$ of $C$ does not coincide with the origin of $\mathbf{R}^4$. Let $K$ be the convex cone generated by $C$ in $\mathbf{R}^4$ (via the embedding we have chosen). By our choice, the interior of this cone is not empty. Then $K$ is locally finitely generated at $V$, but not locally polyhedral at $V$. Indeed, if $S$ is the sphere appearing in Figure 1, the cone $K$ is generated by $V$ and the subcone $K'$ of $K$, generated by the sphere $S$ embedded in $\mathbf{R}^4$. The non-local polyhedrality of $K$ at $V$ is clear.
\end{ex}

\begin{figure}\label{fig2}
\begin{turn}{-60}
\begin{tikzpicture}[scale=0.3]
        \def\R{5}   % R is the radius of the sphere
        \def\r{4.8} % r is the radius of the cone base
        \pgfmathsetmacro\h{sqrt(\R*\R-\r*\r)} 
        \pgfmathsetmacro\a{asin(\r/\R)}
        \pgfmathsetmacro\H{\r*tan(\a)}
        
        %\draw[red,thick] (0,0) -- (-90+\a:\R);
        
        \path   (-\r,-\h) coordinate (S) --++ (\r,0) coordinate (K) --++ (\r,0) coordinate (T)
                (0,0) coordinate (O) --++ (0,-\h-\H) coordinate (U);
        
        % Wireframe
        \draw   (T) arc (-90+\a:270-\a:\R);
        \draw   (S) -- (U) -- (T)
                (T) arc (0:-180:\r cm and 0.2*\r cm);
        \draw [dashed]  (S) -- (T)
                        (K) -- (U)
                        (T) arc (0:180:\r cm and 0.2*\r cm)
                        (S) arc (270-\a:270+\a:\R);
         \node [rotate=60] at (0.5,-19) {$\scaleto{V}{7pt}$};
         \draw [fill] (0,-18) circle [radius=0.1];

    \end{tikzpicture}
    \end{turn}
     \caption{Hyperplane section of a locally finitely generated cone at a point.}
    \end{figure}
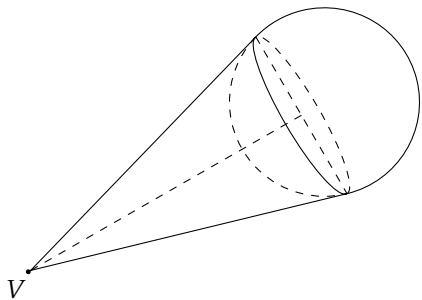
   
We will need the following Lemma, which is the key tool to produce circular parts in $\EffX$, in the proof of Theorem \ref{thmden}.
\begin{lem}[\cite{Kov1993}, Lemma 2.6]\label{lemKov}
Let $Q$ be a smooth compact quadric hypersurface in $\mathbf{R}^k$ and $C\subset \mathbf{R}^k$ a compact convex set. Assume that $Q \not\subset C$, then there exists a non empty open subset $U$ of $Q$  such that $U \subset \partial (\mathrm{Conv}(Q\cup C))$, where $\mathrm{Conv}(Q\cup C)$ is the convex hull of $Q \cup C$. 
\end{lem}

Let $\mathscr{P}\mathrm{eff}(X)$ be the analytic pseudo-effective cone of an IHS manifold $X$, as defined for example in \cite[Section 2.3]{Bouck}. If $X$ is projective
\begin{equation}\label{equality1}
\mathscr{P}\mathrm{eff}(X) \cap N^1(X)_{\mathbf{R}} = \overline{\mathrm{Eff}(X)}
\end{equation}

(cf. \cite[Proposition 1.4]{Dem}). The following result of Boucksom (which holds in a more general context) will be useful.

\begin{thm}[Theorem 3.19 of \cite{Bouck}]\label{thmbouck}
The boundary of the pseudo-effective cone $\mathscr{P}\mathrm{eff}(X)$ of an IHS manifold $X$ is locally polyhedral away from $\overline{\mathscr{BK}_X}$, with extremal rays generated by (the classes of) prime exceptional divisors.
\end{thm}

The following conjecture holds for all known deformation classes of (projective) IHS manifolds (cf. \cite{Mon}, \cite{Mon1} for the O'Grady-type case, \cite{Mat} for the $\mathrm{Kum}_n$-type and K3$^{[n]}$-type case), but it is not known in general.

\begin{conj}[SYZ]\label{conj1}
Let $X$ be a projective IHS manifold and $D$ an integral, isotropic (with respect to the BBF form) divisor on $X$, such that $[D] \in \overline{\mathrm{Mov}(X)}$. Then $\mathscr{O}_X(D)$ induces a rational Lagrangian fibration, i.e.\ there exists a birational map $f\colon X \DashedArrow[->,densely dashed    ] X'$, where $X'$ is another projective IHS manifold, such that $f_*\mathscr{O}_X(D)$ induces a fibration (i.e.\ a surjective morphism with connected fibers) $X' \twoheadrightarrow B$ to a projective $\mathrm{dim}(X)/2$-dimensional base $B$.
\end{conj}

\begin{defn}
Let $X$ be a projective IHS manifold. We define the effective movable cone as $\overline{\mathrm{Mov}(X)}^e:=\overline{\mathrm{Mov}(X)}\cap \mathrm{Eff}(X) $ and $\overline{\mathrm{Mov}(X)}^{+}$ as the convex hull of $\overline{\mathrm{Mov}(X)} \cap N^1(X)_{\mathbf{Q}}$ in $N^1(X)_{\mathbf{R}}$.
\end{defn}

Although the following lemma may be well-known to experts, we did not find it in the literature. As it is a useful result in this context, we decided to include it in the article, with proof.

\begin{lem}\label{lemconeconj}
If $X$ is a projective IHS manifold satisfying Conjecture \ref{conj1}, we have the equality $\overline{\mathrm{Mov}(X)}^{+}=\overline{\mathrm{Mov}(X)}^{e}$. 
\begin{proof}
We start by observing that the interiors of the two involved cones coincide, hence we only have to check that everything works well at the level of the boundaries. Let $\alpha$ be an integral divisor class belonging to $\partial \overline{\mathrm{Mov}(X)}^{+}$. If $q_X(\alpha)>0$, as by \cite[Corollary 3.10]{Huy1} the class $\alpha$ is big, we are done. If $q_X(\alpha)=0$, the SYZ conjecture holds under our assumptions, so the class $\alpha$ is effective, hence the inclusion $\overline{\mathrm{Mov}(X)}^{+}\subseteq \overline{\mathrm{Mov}(X)}^{e}$ is verified. Now, let $\alpha$ be an element of $\overline{\mathrm{Mov}(X)}^e$ lying on its boundary. If $\alpha$ belongs to some wall $[E]^{\perp}$, for some prime exceptional divisor $E$, and $q_X(\alpha)>0$, we are done, because $\overline{\mathrm{Mov}(X)}$ is locally rational polyhedral away from $\partial \overline{\mathscr{C}_X}$ (cf.\ \cite[Corollary 4.8]{Den}). It remains to check the case $0 \neq \alpha \in \partial\overline{\mathscr{C}_X}$ (and so $q_X(\alpha)=0$). We can write $\alpha=\sum_{i=1}^ka_i[D_i]$, where the $D_i$ are prime divisors and the $a_i$ are positive real numbers. Let $W\subset N^1(X)_{\mathbf{R}}$ be the rational subspace spanned by the $[D_i]$. By \cite[Lemma 2.7]{Den}, for every $D_i$ we have $q_X(\alpha,D_i)\geq 0$. Moreover $q_X(\alpha)=0$, whence $q_X(\alpha,D_i)=0$ for any $i$. By linear algebra, as the signature of $q_X$ restricted to $N^1(X)_{\mathbf{R}}$ is $(1,\rho(X)-1)$, it follows that a maximal, totally isotropic linear subspace of $N^1(X)_{\mathbf{R}}$ (with respect to $q_X$) must have dimension $1$. Then $W'=\cap_{i=1}^k \left([D_i]^{\perp} \cap W\right)$ is a rational subspace of dimension 1. Indeed $W'$ contains $\alpha$, and there cannot be other elements linearly independent from $\alpha$, because otherwise, we would have a totally isotropic linear subspace of $N^1(X)_{\mathbf{R}}$ of dimension greater than or equal to $2$. We conclude that $\alpha$ is a multiple of a rational, movable class, and we are done. 
\end{proof}
\end{lem}

\begin{rmk}
Note that:
\begin{itemize}
\item $\overline{\mathrm{Mov}(X)}^{+} = \MovX^e$ for all known projective IHS manifolds.
\item The inclusion $\overline{\mathrm{Mov}(X)}^{+}\supseteq \overline{\mathrm{Mov}(X)}^{e}$ holds for any projective IHS manifold.
\end{itemize}
\end{rmk}

The Kawamata-Morrison movable cone conjecture for projective IHS manifolds predicts the existence of a fundamental domain for the action of $\mathrm{Bir}(X)$ on $\overline{\mathrm{Mov}(X)}^e$, and the theorem below by Markman is a slightly weaker version. 

\begin{thm}[\cite{Mark}, Theorem 6.25]\label{thmmark}
Let $X$ be a projective IHS manifold. There exists a rational convex polyhedral cone $\Pi$ in $\overline{\mathrm{Mov}(X)}^{+}$, such that $\Pi$ is a fundamental domain for the action of $\mathrm{Bir}(X)$ on $\overline{\mathrm{Mov}(X)}^{+}$.
\end{thm}

\begin{rmk}\label{Kawconconj}
If $X$ is a projective IHS manifold belonging to one of the known deformation classes, combining Lemma \ref{lemconeconj} and Theorem \ref{thmmark}, one deduces that $\Pi$ is also a fundamental domain for the action of $\mathrm{Bir}(X)$ on $\overline{\mathrm{Mov}(X)}^{e}$.
\end{rmk}

\subsection{Some hyperbolic geometry}
Consider a real vector space $V$ of dimension $m + 1$ endowed with a nondegenerate quadratic form $q$ of signature $(1, m)$. Thus, abstractly, $(V,q)$ is isomorphic to $\mathbf{R}^{m+1}$ with the quadratic form $x_0^2-x_1^2-\dots-x_m^2$. The set 
\[
\{x \in V | q(x) > 0\}
\] 
has two connected components which are interchanged by $x\to -x$. We usually distinguish one of the two connected components, say $\mathscr{C} \subset V$, and call it the positive cone. Thus,
\[
\{x \in V \;|\; q(x)>0\}=\mathscr{C} \sqcup \left(-\mathscr{C}\right).
\]
We write $\text{O}(V)$ for the orthogonal group $\text{O}(V,q)$, which is abstractly isomorphic to $\text{O}(1,m)$. We denote by $\text{O}^+(V)\subset \text{O}(V)$  the index two subgroup of transformations preserving the positive cone $\mathscr{C}$.
 \vspace{0.2cm}
 
 The subset $\mathscr{C}(1)$ of all $x \in \mathscr{C}$ with $q(x)=1$ is isometric to the hyperbolic $m$-space
\[
\mathbf{H}^m := \{x \in \mathbf{R}^{m+1} \; | \; x_0^2-x_1^2-\dots -x_m^2=1, \; x_0>0\}.
\] 
By writing
$\mathscr{C} \cong \mathscr{C}(1) \times \mathbf{R}_{>0}\cong \mathbf{H}^m \times \mathbf{R}_{>0},$
questions concerning the geometry of $\mathscr{C}$ can be reduced to analogous ones for $\mathbf{H}^m$. We can equivalently view $\mathbf{H}^n$ as the projectivization of $\mathscr{C}$, denoted by $\mathbf{P}(\mathscr{C})$ and called the \emph{projectivized positive cone} (this terminology follows the one adopted by Amerik and Verbitsky in their works). Then we can compactify the hyperbolic space $\mathbf{H}^m$ by adding its points at infinity, namely 
\[
\partial \mathbf{H}^m:=\{x \in \mathbf{P}(V) \; | \; x_0^2-x_1^2-\dots -x_m^2=0\}.
\]
Then $\overline{\mathbf{H}^m}=\mathbf{H}^m \cup \partial \mathbf{H}^m$ is a compactification for $\mathbf{H}^m$. The interior (resp.\ boundary) of $\overline{\mathbf{H}^m}$ is $\mathbf{H}^m$ (resp.\ $\partial \mathbf{H}^m$). For an excellent account of hyperbolic geometry, we refer the reader to \cite{Martelli23}.
 
 \begin{defn}
 Let $\Gamma$ be a non-trivial discrete group of isometries of $\mathbf{H}^m$. The limit set $\Lambda(\Gamma) \subset \partial \mathbf{H}^m$ of $\Gamma$ is the set of all the accumulation points of the orbit $\Gamma\cdot x$ in $\mathbf{H}^m$ of any point $x \in \mathbf{H}^m$.
 \end{defn}
 
 Note that the limit set does not depend on the choice of the point (cf.\ \cite[Exercise 5.1.1]{Martelli23}).
 
\begin{defn}
Let $\Gamma$ be a non-trivial discrete group of isometries of $\mathbf{H}^m$. We say that $\Gamma$ is elementary if its limit set consists of at most two points.
\end{defn}

 Let $X$ be a projective IHS manifold of Picard number $m+1$, and $q_X$ the Beauville-Bogomolov-Fujiki (BBF) form on $X$. The situation above is represented by the couple $(N^1(X)_{\mathbf{R}},q_X)$. It is known that $\mathrm{Mon}^2_{\mathrm{Hdg}}(X)$ acts on $\mathrm{Pic}(X)$ via isometries, and preserves the positive cone $\mathscr{C}_X\cap N^1(X)_{\mathbf{R}}$ (which by abuse of notation we still denote by $\mathscr{C}_X$) of $X$. In particular, we have the natural map
 \[
 \rho \colon \mathrm{Mon}^2_{\mathrm{Hdg}}(X)\to \text{O}^+(\mathrm{Pic}(X)).
 \] 
 Let $\Gamma' \subset \text{O}^+(\mathrm{Pic}(X))$ be the image of $\rho$. Markman proved that $\Gamma'$ is of finite index inside $\text{O}^+(\mathrm{Pic}(X))$ (cf.\ \cite[Lemma 6.23]{Mark}). We denote by $\mathbf{P}(\mathscr{C}_X)$ the projectivized positive cone of $X$. By the discussion above, $\mathbf{P}(\mathscr{C}_X)$ is a finite-dimensional hyperbolic space, and $\Gamma'$ is a non-trivial discrete group of isometries of $\mathbf{P}(\mathscr{C}_X)$.
\subsection{Pell's equations}
A Pell equation is any diophantine equation of the form 
\begin{equation}\label{pelleqn1}
x^2-ky^2=1,
\end{equation}
 where $k$ is a given positive integer, and $k$ is not a square.
The fundamental solution of equation (\ref{pelleqn1}) is the (positive) solution $(x_1,y_1)$ minimizing $x$. Once the fundamental solution is known, all positive solutions (and hence all solutions) can be calculated recursively using the formulas
\begin{equation*}
\begin{aligned}
x_{j+1}&=x_1x_j+ky_1y_j  \\
y_{j+1}&=x_1y_j+y_1x_j.
\end{aligned}
\end{equation*}
These recursive formulas are derived from \cite[Proposition 17.5.2]{Ireland}. This is all we need from the theory of Diophantine equations. For an account, we refer the reader to \cite[Chapter 17]{Ireland}.

\section{Proof of the main Theorem and consequences}\label{Section2}

In this section, we prove the main theorem of the article, namely Theorem \ref{thmden}. We start by deducing Corollary \ref{cor1}, which characterizes the pseudo-effective cone of a projective IHS manifold of Picard number 2, from Oguiso's theorem \cite[Theorem 1.3, item \textit{(2)}]{Ogu}.

\begin{proof2}
%We first observe that $\EffX$ cannot have both one rational extremal ray and one irrational extremal ray. Indeed, if $X$ does not contain any prime exceptional divisor then $\EffX=\MovX$ (cf.\ \cite[Corollary 3.5]{Den}) and we are done by Theorem \ref{thmogu}. As by Theorem \ref{thmbouck} and equality (\ref{equality1}) the classes of the prime exceptional divisors span extremal rays in $\EffX$, we only have to check the case "$X$ contains only one exceptional prime divisor". In such case $\MovX$ is rational by Theorem \ref{thmogu} and \cite[Theorem 6.17]{Mark}, and the extremal ray of $\MovX$ spanned by an integral class of BBF square $0$ is also extremal for $\EffX$ and we are done. 
By Theorem \ref{thmogu} and the duality  $\EffX={\MovX}^{*}$ (cf.\ \cite[Lemma 2.7]{Den}), either both the extremal rays of $\EffX$ are irrational and $\mathrm{Bir}(X)$ is infinite, or rational and $\mathrm{Bir}(X)$ is finite.
\end{proof2}

\begin{rmk}
We observe that under the assumptions of Corollary \ref{cor1} if $X$ satisfies Conjecture \ref{conj1} (for example if $X$ belongs to one of the known deformation classes), whenever $\EffX$ is rational we have $\EffX=\mathrm{Eff}(X)$, i.e. both the extremal rays of the pseudo-effective cone are spanned by the classes of some integral, effective divisors. Also, note that both the cases in Corollary \ref{cor1} do occur (see for example \cite[Lemma 3.2]{Ulrike} and \cite[Proposition 5.3]{Ogu}).
\end{rmk}

\begin{proof1}\label{proof1}

\underline{Suppose that $\mathrm{Neg}(X)=\emptyset$}. In this case $\overline{\mathrm{Eff}(X)}=\overline{\mathscr{C}_X}$, hence $\EffX$ is circular.

\underline{Suppose that $\mathrm{Neg}(X) \neq \emptyset$}. 
We first show that $\overline{\mathrm{Eff}(X)}$ does not contain any circular part. Let $E$ be any prime exceptional divisor on $X$. We recall that $\Gamma'$ is the image of
\[
 \rho \colon \mathrm{Mon}^2_{\mathrm{Hdg}}(X)\to \text{O}^+(\mathrm{Pic}(X)).
 \] 
 The orbit of $E$ under the action of $\Gamma'$ is made (up to a sign) of stably exceptional divisors (i.e. $\mathscr{O}_X(E)$ is a stably exceptional line bundle), and the walls $\{g(E)^{\perp}\}_{g\in \Gamma'}$ intersect $\mathscr{C}_X$.
Indeed, the BBF square of $E'$ is negative, where $E'=g(E)$ for some $g \in \Gamma'$. Let $y$ be any class in $\mathscr{C}_X$. Consider the class
\[
y':=y-\frac{q_X(E',y)}{q_X(E')}[E'].
\]
Clearly, $y'\in (E')^{\perp}$, and $q_X(y')=q_X(y',y)>0$ (cf. \cite[Proposition 2.14]{Den}), thus $y'\in E'^{\perp} \cap \mathscr{C}_X$. Suppose by contradiction that $\overline{\mathrm{Eff}(X)}$ contains a circular part $\mathcal{C}$. We can assume that $\mathcal{C}= \mathbf{R}^{>0} \mathcal{C}$. By Theorem \ref{thmbouck} and equality (\ref{equality1}), the circular part $\mathcal{C}$ must be contained in $\partial \MovX \cap \partial \EffX$. Also, $\MovX$ is locally rational polyhedral away from the boundary of $\mathrm{Big}(X)$ (cf. \cite[Corollary 4.8]{Den}), hence $\mathcal{C} \subset \partial \overline{\mathscr{C}_X} \cap \partial \MovX$. Let $x$ be a point lying in $E^{\perp}\cap \partial\overline{\mathscr{C}_X}$. The quotient $\mathbf{P}(\mathscr{C}_X)/\Gamma'$ is a hyperbolic manifold, and since $\Gamma'$ is of finite index in $\text{O}^+(\mathrm{Pic}(X))$ and $\rho(X)\geq 3$, it is of finite volume (cf. \cite[Section 3.2]{article}). It follows by \cite[Proposition 5.1.9]{Martelli23} that $\Gamma'$ is non-elementary, and so, by \cite[Theorem 1.29]{Seade}, $\Gamma'$ acts with dense orbits on $\partial \mathbf{P}(\mathscr{C}_X)$. The rays $\{\mathbf{R}^{> 0}(g(x))\}_{g\in \Gamma'}$ are contained in the boundary of the positive cone. Then the set of rays $\mathbf{R}^{> 0}(g(x))$ (with $g\in \Gamma'$) contained in $\mathcal{C}$ form a dense subset of the circular part $\mathcal{C}$, and any of them is orthogonal to the corresponding $g(E)$. This implies the existence of a class lying in $\mathrm{Int}(\overline{\mathrm{Mov}(X)})$, $q_X$-intersecting trivially some stably exceptional divisor, and this is a contradiction because any element in the interior of the movable cone intersects any stably exceptional divisor positively (cf. \cite[Proposition 6.10]{Mark}). Then $\overline{\mathrm{Eff}(X)}$ has no circular part.
 Now, using that $\EffX$ does not contain any circular part, we show the equality $\EffX=\overline{\sum_{E \in \mathrm{Neg}(X)}\mathbf{R}^{\geq 0} [E]}$. To do so, as $\EffX$ does not contain lines (see for example \cite[Proposition 3.2]{Den}), by Minkowski's Theorem on closed convex cones not containing lines, it suffices to show that $\overline{\sum_{E \in \mathrm{Neg}(X)}\mathbf{R}^{\geq 0} [E]}$ contains all the extremal rays of $\EffX$. Suppose by contradiction that there exists an extremal ray $R$ of $\EffX$ not belonging to $\overline{\sum_{E \in \mathrm{Neg}(X)}\mathbf{R}^{\geq 0} [E]}$. Pick a non-zero element $v \in R$ (which must therefore  satisfy $q_X(v)=0$) and consider the hyperplane
\[
H:=\{x \in N^1(X)_{\mathbf{R}} \; | \; q_X(x,A)=q_X(v,A)\},
\] 
where $A$ is any ample divisor. Clearly, $Q:=H \cap \partial \overline{\mathscr{C}_X}$ is a smooth compact quadric hypersurface in $H$. Indeed $q_X(v,A)>0$ (see for example \cite[Proposition 3.2]{Den})), hence $H$ does not intersect $-\partial\overline{\mathscr{C}_X}$, so that $H \cap \partial \overline{\mathscr{C}_X}$ is non-degenerate. Moreover, up to a rescaling, any element $y \in \partial\overline{\mathscr{C}_X}$ is such that $q_X(y,A)=q_X(v,A)$, hence $Q$ must be compact. Let $N$ be the closure of the convex hull in $H$ of all the elements $y\in H$ belonging to an extremal ray of the type $\mathbf{R}^{\geq 0} [E]$, where $E$ is a prime exceptional divisor of $X$. The subset $N$ of $H$ is compact and convex. As $v \in Q \setminus N$, by Lemma \ref{lemKov}, there exists an open subset $U$ of $Q$ contained in $\partial(\mathrm{Conv}(Q \cup N))$, where $\mathrm{Conv}(Q \cup N)$ is the convex hull of $Q \cup N$. Then $\mathbf{R}^{\geq 0}U$ is a circular part of $\EffX$, and this contradicts the fact that $\EffX$ does not contain circular parts.
\end{proof1}

\begin{rmk}
For a somehow similar statement about the nef cone see \cite{Mat18}.
\end{rmk}

\begin{cor}\label{cor4}
Let $X$ be a projective IHS manifold of Picard number at least $3$, carrying a prime exceptional divisor. Then $\mathrm{card}(\mathrm{Neg}(X))\geq \rho(X)$ and $\EffX$ is rational polyhedral if and only if $\mathrm{card}(\mathrm{Neg}(X))$ is finite.
\begin{proof}
By Theorem \ref{thmden} we have, $\overline{\mathrm{Eff}(X)}=\overline{\sum_{E \in \mathrm{Neg}(X)}\mathbf{R}^{\geq 0}[E]}$, and the latter equality is satisfied only if $\#\mathrm{Neg}(X)\geq \rho(X)$, because the interior of $\overline{\sum_{E \in \mathrm{Neg}(X)}\mathbf{R}^{\geq 0}[E]}$ (which is non-empty) equals the interior of $\sum_{E \in \mathrm{Neg}(X)}\mathbf{R}^{\geq 0}[E]$. The proof of the second part of the statement is clear.
\end{proof}
\end{cor}

\begin{ex}
Let $S$ be a projective K3 surface of Picard number $2$, carrying one (and only one) smooth rational curve. Then, by Corollary \ref{cor4}, $S^{[n]}$ carries at least one prime exceptional divisor coming neither from $S$, nor from the desingularization of $S^{(n)}$, for any $n$.
\end{ex}

\begin{cor}
Let $X$ be a projective IHS manifold of Picard number greater than 3, carrying a prime exceptional divisor. Then the nef cone $\mathrm{Nef}(X)$ has no circular part. 
\end{cor}
\begin{proof}
By Theorem \ref{thmden}, $\overline{\mathrm{Mov}(X)}$ has not circular parts, and $\mathrm{Nef}(X)$ is rational polyhedral away from the boundary of $\mathscr{C}_X$, thanks to a result of Kawamata (cf. \cite[Theorem 5.7]{Kaw}), hence also $\mathrm{Nef}(X)$ does not have any circular part. 
\end{proof}

In the proof of Corollary \ref{thm2} we will need a theorem by Burnside and a result of Oguiso.

\begin{thm}[\cite{Burn}, Main Theorem]\label{burnthm}
Let $G$ be a subgroup of $\mathrm{GL}(n,\mathbf{C})$. Assume that there exists a positive integer $d$
such that any element of $G$ has order at most $d$. Then $G$ is a finite group.
\end{thm}

\begin{lem}[\cite{Ogu}, Corollary 2.6]\label{lemogu}
Let $Y$ be a smooth projective variety with trivial canonical bundle, such that $h^1(Y,\mathscr{O}_Y)=0$. Moreover, let $r\colon \mathrm{Bir}(Y) \to \mathrm{GL}(N^1(Y)_{\mathbf{R}})$ be the natural representation of $\mathrm{Bir}(Y)$. Then the following hold.
\begin{enumerate}
\item $[\mathrm{Bir}(Y):\mathrm{Aut}(Y)]=[r(\mathrm{Bir}(Y)):r(\mathrm{Aut}(Y))]$.
\item If $G$ is a subgroup of $\mathrm{Bir}(Y)$, $G$ is finite if and only if there is a positive
integer $d$ such that every element of $r(G)$ as order at most $d$.
\end{enumerate}
\end{lem}
\begin{proof3}
We start by proving \textit{(1)} $\Leftrightarrow$ \textit{(2)}. If $\rho(X)=2$ we are done, by Corollary \ref{cor1}. If $\rho(X)\geq 3$, consider the natural representation 
\[
f \colon \mathrm{Bir}(X) \to \mathrm{GL}(N^1(X)_{\mathbf{R}}). 
\]

 We first prove \textit{(1)} $\Rightarrow$ \textit{(2)}. As $\mathrm{Eff}(X)$ is rational polyhedral, $X$ carries a finite number $N$ of prime exceptional divisors, by Corollary \ref{cor4}. In particular, $\mathrm{Eff}(X)$ does not contain circular parts, hence, by Theorem \ref{thmden}, the equality
\[
\mathrm{Eff}(X)=\sum_{E\in \mathrm{Neg}(X)} \mathbf{R}^{\geq 0}[E]
\] 
 holds in this case. If $s$ is a birational self-map of $X$ and $f(s)$ is the automorphism induced on $N^1(X)_{\mathbf{R}}$, we observe that $f(s)$ has order at most $N!$. Indeed, $f(s)$ acts on $\mathrm{Neg}(X)$, hence an automorphism of $N^1(X)_{\mathbf{R}}$ is uniquely determined by $f(s)(\mathrm{Neg}(X))$, because in $\mathrm{Neg}(X)$ we can find a basis for the Néron-Severi space. Now we use Theorem \ref{burnthm} to conclude that $\Bir$ is finite. 
 
 Now we prove \textit{(2)} $\Rightarrow$ \textit{(1)}. As $\Bir$ is finite, $\MovX^e$ is rational polyhedral, by Theorem \ref{thmmark} and Remark \ref{Kawconconj}. Note that the equalities $\MovX^e=\mathrm{Mov}(X)=\MovX$ hold in this case. Indeed, it suffices to note that since $\MovX^e$ is rational polyhedral and Conjecture \ref{conj1} is satisfied by assumption, $\MovX^e$ is spanned by movable integral divisors. Then, as $\EffX=\MovX^{*}$ (cf.\ \cite[Lemma 2.7]{Den}), also $\EffX$ is rational polyhedral. To conclude the proof, we observe that $\EffX=\mathrm{Eff}(X)$. Indeed, since $\EffX$ is rational polyhedral, it has finitely many extremal rays, and these are spanned by effective integral divisor classes. We now prove \textit{(2)} $\Leftrightarrow$ \textit{(3)}. Let $\text{O}^+(N^1(X))$ be the group of isometries of $N^1(X)$ preserving the positive cone $\mathscr{C}_X$.  Clearly $\text{O}^+(N^1(X))$ has index $2$ in $\text{O}(N^1(X))$. Moreover, by \cite[Lemma 6.23]{Mark}, the image of $\mathrm{Mon}^2_{\mathrm{Hdg}}(X)$ via the natural map $\rho \colon \mathrm{Mon}^2_{\mathrm{Hdg}}(X) \to \text{O}^+(N^1(X))$ has finite index in $\text{O}^+(N^1(X))$ and is isomorphic to a semi-direct product of $W_{\mathrm{Exc}}$ and $\mathrm{Mon}^2_{\mathrm{Bir}}(X)/\mathrm{ker}(\rho)$. 
Consider the chain of inclusions
\begin{equation}\label{chain}
W_{\mathrm{Exc}} \subset \mathrm{Im}(\rho) \subset \text{O}^+(N^1(X))\subset \text{O}(N^1(X)),
\end{equation}
 If $\mathrm{Bir}(X)$ is finite, then  $\mathrm{Mon}^2_{\mathrm{Bir}}(X)/\mathrm{ker}(\rho)$ is finite, whence $\text{O}(N^1(X))/W_{\mathrm{Exc}}$ is finite, and this proves \textit{(2)} $\Rightarrow$ \textit{(3)}. Now, suppose that $\text{O}(N^1(X))/W_{\mathrm{Exc}}$ is finite. Then $\mathrm{Mon}^2_{\mathrm{Bir}}(X)/\mathrm{ker}(\rho)$ is finite. The kernel of $\rho$ is finite by \cite[Proposition 2.4]{Ogu}. In particular, the set of birational self-maps of $X$ acting trivially on $N^1(X)$ is finite. This implies that $\mathrm{Mon}^2_{\mathrm{Bir}}(X)$ is finite, whence $\mathrm{Bir}(X)$ is finite, and the implication \textit{(3)} $\Rightarrow$ \textit{(2)} is proven.

The implication \textit{(3)} $\Rightarrow$ \textit{(4)} is obtained by \textit{(3)} $\Rightarrow$ \textit{(2)} $\Rightarrow$ \textit{(1)} $\Rightarrow$ \textit{(4)}. Suppose now that $X$ carries a prime exceptional divisor and that its Picard number is at least $3$. 

The implication \textit{(4)} $\Rightarrow$ \textit{(1)} is nothing but the "only if" part of Corollary \ref{cor4}, thus all the statements are equivalent in this case because we also have \textit{(4)} $\Rightarrow$ \textit{(1)}$\Rightarrow$ \textit{(2)} $\Rightarrow$ \textit{(3)}.

\end{proof3}

\begin{rmk}\label{rmknefcone}
Note that, given a projective IHS manifold $X$ belonging to one of the known deformation classes, replacing $\mathrm{Bir}(X)$ with $\mathrm{Aut}(X)$ and the Kawamata-Morrison movable cone conjecture with the classical Kawamata-Morrison cone conjecture, stated for $\mathrm{Nef}^e(X)=\mathrm{Nef}(X) \cap \mathrm{Eff}(X)$ (see \cite[Theorem 1.3]{Mark1} and \cite[Section 5.2]{Amerik}), a similar argument to that given in the proof of \textit{(1)} $\Leftrightarrow$ \textit{(2)} in Corollary \ref{thm2} can be used to prove that $\mathrm{Nef}(X)$ is rational polyhedral if and only if $\mathrm{Aut}(X)$ is finite. 
\end{rmk}

We recall the following definition.
\begin{defn}\label{DefMDS}
Let $X$ be a normal and $\mathbf{Q}$-factorial projective variety. We say that $X$ is a Mori dream space (MDS) if the following properties hold:
\begin{enumerate}

\item $\mathrm{Pic}(X)$ is finitely generated (equivalently, $h^1(\mathscr{O}_X)=0$);
\item $\mathrm{Nef}(X)$ is generated by the classes of finitely many semiample divisors;
\item there is a finite collection of small $\mathbf{Q}$-factorial modifications $s_i \colon X \DashedArrow[->,densely dashed    ] X_i$, such that every $X_i$ satisfies items $(1), (2)$ and  
\[
\mathrm{Mov}(X)=\bigcup_{i} s_i^{*}(\mathrm{Nef}(X_i)).
\]
\end{enumerate}
\end{defn}

The corollary below explains when a projective IHS manifold is a MDS.

\begin{cor}
Let $X$ be a projective IHS manifold with $b_2(X)\neq 4$ and satisfying Conjecture \ref{conj1}. Then the following conditions are equivalent:
\begin{enumerate}
\item $X$ is a MDS,
\item $\mathrm{Bir}(X)$ is finite,
\item $0<\rho(X)<3$ and $\mathrm{Eff}(X)$ is rational polyhedral, or $\rho(X) \geq 3$ and $\mathrm{Neg}(X)$ is non-empty and finite.
\end{enumerate}
\begin{proof}
\textit{(1)} $\Leftrightarrow$ \textit{(2)}. Suppose that $X$ is a MDS. Then, up to automorphisms, there are finitely many small $\mathbf{Q}$-factorial modifications $X \DashedArrow[->,densely dashed] X$. Moreover, by Remark \ref{rmknefcone}, $\mathrm{Aut}(X)$ is finite, as by definition of MDS $\mathrm{Nef}(X)$ is rational polyhedral. Thus $\mathrm{Bir}(X)$ is finite. Suppose now that $\mathrm{Bir}(X)$ is finite. By \cite[Theorem 3.17]{Amerik} the BBF square of the integral, primitive, and extremal classes of the Mori cones of the birational models of $X$ is bounded, hence by \cite[Corollary 1.5]{Mark1} the number of birational models of $X$ is finite (up to isomorphism). This implies that $(3)$ in Definition \ref{DefMDS} is satisfied. The base point free Theorem, the hypothesis that $X$ verifies Conjecture \ref{conj1}, and Remark \ref{rmknefcone} imply that also $(2)$ in Definition \ref{DefMDS} is satisfied. Thus $X$ is a MDS. 

\textit{(2)} $\Leftrightarrow$ \textit{(3)}. Follows directly from Corollary \ref{thm2}.
\end{proof}
\end{cor}

We conclude this section by proving Corollary \ref{corHoring}. For results about the existence of ample uniruled divisors on some primitive symplectic varieties see \cite{BG22} and \cite{LMP23}. 

\begin{proof4}
As $Y$ is $\mathbf{Q}$-factorial, any irreducible component of the exceptional locus of $f$ is divisorial. By Corollary \ref{cor4}, $\mathrm{card}(\mathrm{Neg}(X))\geq \rho(X)$. If $A$ is any ample Cartier divisor on $Y$, we have $\mathrm{Exc}(f)=\mathbf{B}_+(f^{*}(A))$ (cf.\ \cite[Proposition 2.3]{Pac}). In particular, by \cite[Corollary 2.15]{Den}, \cite[Proposition 2.4]{DenOrtiz}, the Gram-matrix of the irreducible components of $\mathrm{Exc}(f)$ is negative definite, and so the classes in $N^1(X)_{\mathbf{R}}$ of the irreducible components of $\mathrm{Exc}(f)$ are linearly independent. In particular, in $\mathrm{Exc}(f)$ there are at most $\rho(X)-1$ irreducible components. On the other hand, $\mathrm{card}(\mathrm{Neg}(X))\geq \rho(X)$, hence there exists a prime exceptional divisor $E$ on $X$ which is not contracted by $f$. Then $f_*(E)$ is a uniruled divisor on $Y$.
\end{proof4}

By adopting the same strategy to prove Theorem \ref{thmden}, we obtain the following version of Theorem \ref{thmden} in the singular setting.
\begin{thm}\label{thmden3}
Let $X$ be a projective $\mathbf{Q}$-factorial primitive symplectic variety with terminal singularities of Picard number at least 3. Then either $\overline{\mathrm{Eff}(X)}=\overline{\mathscr{C}_X}$, or \[
\overline{\mathrm{Eff}(X)}=\overline{\sum_E\mathbf{R}^{\geq 0}[E]},
\]
where the sum runs over the prime exceptional divisors of $X$.
\begin{proof}
All the general theory about primitive symplectic varieties that is needed to prove the theorem is contained in the papers \cite{BL22}, \cite{KMPP19}, \cite{LMP22}, \cite{LMP23}, to which we refer the reader. Below we provide an outline of the proof, following the proof of Theorem \ref{thmden} and highlighting the steps we consider most important.
If there are no prime exceptional divisors, $\overline{\mathrm{Eff}(X)}=\overline{\mathscr{C}_X}$ (here we used the singular version of \cite[Corollary 3.5]{Den}), and we are done. Now, suppose that $X$ contains a prime exceptional divisor $E$. Furthermore, suppose by contradiction that $\overline{\mathrm{Eff}(X)}$ contains a circular part $\mathcal{C}$. One can show that $\overline{\mathrm{Mov}(X)}$ is locally rational polyhedral away from the boundary of $\overline{\mathscr{C}_X}$ (more precisely, one only needs to adapt \cite[Corollary 4.8]{Den} to the singular setting). Hence, $\mathcal{C}$ must be contained in $\partial \overline{\mathrm{Mov}(X)}\cap \partial \overline{\mathscr{C}_X}$. Let $x$ be a point lying in $E^{\perp} \cap \partial \mathscr{C}_X$. Any prime exceptional divisor stays (up to a sign) stably exceptional under the action of the monodromy group $\mathrm{Mon}^{2,\mathrm{lt}}_{\mathrm{Hdg}}(X)$ (cf. \cite[Proposition 5.3, item (2)]{LMP22}). The image of $\mathrm{Mon}^{2,\mathrm{lt}}_{\mathrm{Hdg}}(X)$ in $\text{O}^+(N^1(X))$ (which we denote by $\Gamma'$) is of finite index by \cite[Lemma 6.3]{LMP22} and \cite[Theorem 1.2, item (1)]{BL22}. Moreover, the signature of the singular version of the BBF quadratic form on $N^1(X)$ is $(1,\rho(X)-1)$. Then $\mathbf{P}(\mathscr{C}_X)/\Gamma'$ is a hyperbolic manifold of finite volume (argue as in \cite[Section 3.2]{article}). We obtain a contradiction by arguing as in the smooth case and using the characterization of $\mathrm{int}(\mathrm{Mov}(X))$ provided in \cite[Proposition 5.9]{LMP22}. To show the equality $\overline{\mathrm{Eff}(X)}=\overline{\sum_E\mathbf{R}^{\geq 0}[E]}$, we argue as in the smooth case.
\end{proof}
\end{thm}

 As a consequence of the theorem above, we also obtain a more general statement for Corollary \ref{corHoring}.

\begin{cor}\label{corsingular}
Let $X$ be a projective $\mathbf{Q}$-factorial primitive symplectic variety with non-terminal singularities of Picard number at least 2. Then $X$ carries a prime exceptional (hence uniruled) divisor.
\begin{proof}
Let $X'\to X$ be a $\mathbf{Q}$-factorial terminalization of $X$, which exists by \cite[Corollary 1.4.3]{BCHM10}. Then $\rho(X')\geq 3$, and by Theorem \ref{thmden3} $X'$ carries at least $\rho(X')\geq 3$ prime exceptional divisors. To conclude the proof one argues as in the proof of Corollary \ref{corHoring}.
\end{proof}
\end{cor}

\section{Producing effective divisors}\label{Section3}

In this section, we show how to construct explicit effective integral divisors with some fixed monodromy invariants. We first need the following technical lemma, which is an adaptation of \cite[Lemma 3.1]{Kov} to our case.

\begin{lem}\label{tecnolemma}
Let $X$ be a projective IHS manifold, $D$ an integral divisor such that $[D] \in \mathscr{C}_X$ and $\overline{E}$ an integral divisor such that $0 \neq [\overline{E}] \in \partial \EffX$. Let $t:=\mathrm{div}(\overline{E})$ be the divisibility of $\overline{E}$. Consider the 2-plane $\pi:=\langle [\overline{E}],[D] \rangle \subset N^1(X)_{\mathbf{R}}$.
\begin{enumerate}[label=(\alph*)]
\item If $q_X(\overline{E})=0$ there exists a pseudo-effective class $\alpha \in \pi \cap \EffX$ represented by an integral divisor such that $q_X(\alpha)=0$ and $\alpha$ and $[\overline{E}]$ are on opposite sides of $[D]$.
\item If $q_X(\overline{E})=te<0$ there exists an integral class $\alpha$, such that $q_X(\alpha)=0$ or $q_X(\alpha)=te$ and $\alpha$ and $[\overline{E}]$ are on opposite sides of $\mathbf{R}^{\geq 0}[D]$. Furthermore, only a positive multiple of the divisor defining $\alpha$ can be effective.
\end{enumerate}
\begin{proof}
Set $d=q_X(D)$, $bt=q_X(\overline{E},D)$ and $te=q_X(\overline{E})$. Note that $bt>0$, for example by \cite[Lemma 3.1]{Mark1}.

\textit{(a)} Let $\alpha=x[D]-y[\overline{E}]$ be an element of $\pi$, for $x,y$ real numbers. Then in this case $q_X(\alpha)=dx^2-2btxy$, and by choosing $x=2bt$ and $y=d$, we obtain $\alpha=2bt[D]-d[\overline{E}]$ satisfying $q_X(\alpha)=0$. Also, $q_X(\alpha,D)=btd>0$, hence, by \cite[Lemma 3.1]{Mark1}, $\alpha$ belongs to $\overline{\mathscr{C}_X}$ and so is pseudo-effective. As $x[D]=\alpha+y[\overline{E}]$, clearly $\alpha$ and $[\overline{E}]$ are on opposite sides of $[D]$.

\textit{(b)} Let $\alpha=tx[D]-y[\overline{E}]$ be an element of $\pi$, for $x,y$ real numbers. Then 
\[
q_X(\alpha)=t^2x^2d+y^2et-2xyt^2b=te\left(\frac{tx^2d}{e}+y^2-2xyt\frac{b}{e}\right).
\] 
If we set $x'=y-\frac{xtb}{e}$, $y'=-\frac{x}{e}$ and $N=t^2b^2-tde$, we can rewrite $q_X(\alpha)$ as 
\begin{equation}\label{squarealpha}
q_X(\alpha)=te\left[\left(y-\frac{xtb}{e}\right)^2+tde\frac{x^2}{e^2}-t^2b^2\frac{x^2}{e^2}\right]=te[(x')^2-N(y')^2].
\end{equation}

 \underline{If $N$ is a square}, then $q_X(\alpha)=te\left(x'-\sqrt{N}y'\right)\left(x'+\sqrt{N}y'\right)$. Choosing $x'=N,y'=\sqrt{N}$, i.e. $x=-e\sqrt{N},y=N-\sqrt{N}tb$, we obtain the element $\alpha=-te\sqrt{N}[D]-(N-\sqrt{N}tb)[\overline{E}]$ satisfying $q_X(\alpha)=0$. We observe that $N-\sqrt{N}tb>0$. If $b=0$ or $b<0$ this is trivial. If $b>0$ and we assume $N-\sqrt{N}tb\leq 0$, we would have $N^2-N(tb)^2\leq 0$, which  would imply $N-(tb)^2 \leq 0$ and this is a contradiction, because we have $N-(tb)^2=-tde>0$.
 Also, we note that the chosen $\alpha$ belongs to $\overline{\mathscr{C}_X}$. Indeed, it suffices to pick an element $\beta \in \MovX \cap \mathscr{C}_X \cap [\overline{E}]^{\perp}$ and to observe that $q_X(\beta,\alpha)=-\sqrt{N}teq_X([D],\beta)>0$. For instance, let $D'$ be an integral divisor whose class lies in $\mathrm{int}\left(\mathrm{Mov}(X)\right)$. One can choose $\beta:=[D']-\frac{q_X(D',\overline{E})}{q_X(E)}[\overline{E}]$. Clearly $q_X(\beta,\overline{E})=0$ and 
 \[ 
 q_X(\beta)=\frac{q_X(D')q_X(\overline{E})-q_X(D',\overline{E})^2}{q_X(\overline{E})}>0
 \] 
 (for the latter inequality see for example \cite[Proposition 2.15]{Den}). Also in this case $\alpha$ and $[\overline{E}]$ are on opposite sides of $[D]$. 
 
 \underline{If $N$ is not a square}, the Pell equation 
 \begin{equation}\label{pelleqn}
 x'^2-Ny'^2=1 
 \end{equation} has infinitely many integer solutions (see \cite[Proposition 17.5.2]{Ireland}) and we may choose a solution $(x',y')$ of positive integers. Arguing as above we obtain a class 
 \begin{equation}\label{alpha}
 \alpha=-tey'[D]-(x'-tby')[\overline{E}]
 \end{equation}
 satisfying $q_X(\alpha)=te$. Also in this case we have $x'-tby'>0$. Indeed, if $b=0$ or $b<0$ this is trivial. If $b>0$ and we assume $x'-tby'\leq 0$, we would have $x'^2-(tb)^2y'^2=1-tdey'^2\leq 0$, which is clearly a contradiction. Also, as above, $\alpha$ and $[\overline{E}]$ are on opposite sides of $[D]$. Now, suppose that a multiple of the divisor defining $\alpha$ (namely $D':=-tey'D-(x'-tby')\overline{E}$) is effective. By contradiction, without loss of generality, we can assume that $-D'$ is effective. Then $-D'-tey'D=(x'-tby')\overline{E}$ is big and this is a contradiction because the latter belongs to the boundary of $\mathrm{Big}(X)$. It follows that only a positive multiple of $D'$ can be effective, and in such case the class $\alpha$ is effective.
\end{proof}
\end{lem}

\begin{center}
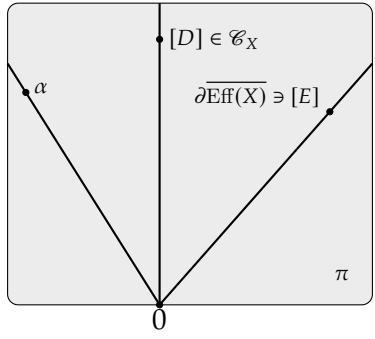
\begin{figure}\label{fig1}
\begin{tikzpicture}[scale=0.8]
\draw [rounded corners, fill=white!70!lightgray, thin] (0,0) rectangle (6,5);
\draw [fill] (2.5,0) circle [radius=0.05];
\draw [thick] (2.5,0) -- (6,4);
\draw [thick] (2.5,0) -- (2.5,5);
\draw [thick] (2.5,0) -- (0,4);
\node at (2.5,-0.25) {0};
\draw [fill] (2.5,4.4) circle [radius=0.05];
\draw [fill] (0.3,3.52) circle [radius=0.05];
\draw [fill] (5.3,3.2) circle [radius=0.05];
\node at (3.4,4.4) {$\scaleto{[D]\in \mathscr{C}_X}{8pt}$};
\node at (4.1,3.5) {$\scaleto{\partial\EffX \ni [\overline{E}]}{10pt}$};
\node at (0.55,3.6) {$\scaleto{\alpha}{4pt}$};
\node at (5.5,0.5) {$\scaleto{\pi}{4pt}$};
\end{tikzpicture}
\caption{Picturing Lemma \ref{tecnolemma}}
\end{figure}
\end{center}

\begin{proof5}
Let $\overline{E}$ be an integral divisor representing $\beta$, so that $[\overline{E}]=\beta$ and $[E]=k[\overline{E}]$. Also, let $D$ be any integral divisor whose class lies in $\mathscr{C}_X$. If the class $\alpha$ constructed in item \textit{(b)} of Lemma \ref{tecnolemma} with respect to $[\overline{E}]$ and $[D]$ is isotropic, then it is effective, by Theorem \ref{thmbouck} and because Conjecture \ref{conj1} holds for the known deformation classes of IHS manifolds. Suppose then that $\alpha$ is not isotropic. We will show that up to choosing a suitable solution for the Pell equation (\ref{pelleqn}), the class $\alpha$ is primitive. Then, since $[\overline{E}]$ and $\alpha$ are primitive, we will use Eichler's criterion to infer that $\alpha$ is stably exceptional, whence effective. We adopt the notation of Lemma \ref{tecnolemma}. It is natural to split the proof by distinguishing the deformation types.
\vspace{0.2cm}

$\bullet$ Suppose that $X$ is of OG10-type. By \cite[Proposition 3.1]{Mon} we can have $q_X(\overline{E})=-2$ and $\mathrm{div}(\overline{E})=1$ or $q_X(\overline{E})=-6$ and $\mathrm{div}(\overline{E})=3$. In the first case, the class $\alpha$ we get is of BBF square $-2$ and this implies that $\alpha$ is primitive. Furthermore, its divisibility must be one, because by Remark \ref{rmklattice} $\mathrm{div}(\alpha)$ divides  $|A_X|$, and the discriminant group of $X$ is $A_X\cong \mathbf{Z}/3\mathbf{Z}$ (the reader is referred to \cite{Rap} for the computation of the discriminant group of the known deformation classes). %(cf. \cite{Mark5}). 
It follows by \cite[Proposition 3.1]{Mon} and item \textit{(b)} of Lemma \ref{tecnolemma} that the class $\alpha$ is stably exceptional, hence effective. In the second case, we have $q_X(\alpha)=-6$, and this again implies the primitivity of $\alpha$. Also, $\alpha=6y'[D]-(x'-3by')[\overline{E}]$, and for any $\gamma \in H^2(X,\mathbf{Z})$ we have $q_X(\alpha,\gamma)\equiv 0 \text{ mod } 3$. But $|A_X|=3$, hence $\mathrm{div}(\alpha)=3$. Using again \cite[Proposition 3.1]{Mon} and item \textit{(b)} of Lemma \ref{tecnolemma}, we conclude that also in this case the class $\alpha$ is effective. Moreover, $\alpha$ has the same monodromy orbit as that of $\overline{E}$, because by \cite[Theorem]{Onorati22} the monodromy of OG10-type IHS manifolds is maximal.
\vspace{0.2cm}

$\bullet$ Suppose that $X$ is of OG6-type. By \cite[Proposition 6.8]{Mon1} we have either $q_X(\overline{E})=-4$ and $\mathrm{div}(\overline{E})=2$, or $q_X(\overline{E})=-2$ and $\mathrm{div}(\overline{E})=2$. In both cases, the class $\alpha$ is primitive. Furthermore, in the first case we have $\alpha=4y'[D]-(x'-2by')[\overline{E}]$, so that $q_X(\alpha,\gamma) \equiv 0 \; \mathrm{mod} \; 2$, where $\gamma$ is any element of $H^2(X,\mathbf{Z})$. It follows that $\mathrm{div}(\alpha)\geq 2$, and as $A_X \cong \mathbf{Z}/2\mathbf{Z} \times \mathbf{Z}/2\mathbf{Z}$, by Remark \ref{rmklattice}, we conclude that $\mathrm{div}(\alpha)=2$. The same argument proves that also in the second case, the divisibility of $\alpha$ is $2$. By \cite[Proposition  6.8]{Mon1} and item \textit{(b)} of Lemma \ref{tecnolemma} we conclude that the class $\alpha$ is effective. Note also that $\alpha$ has the same monodromy orbit as that of $\overline{E}$, because by \cite[Theorem 5.5(1)]{Mon1} the monodromy of OG6-type IHS manifolds is maximal.
\vspace{0.2cm}

$\bullet$ Suppose that $X$ is of K3$^{[n]}$-type. By \cite[Theorem 1.11]{Mark2} we can have $\mathrm{div}(\overline{E})=2(n-1)$ (resp.\ $\mathrm{div}(\overline{E})=n-1$), and $q_X(\overline{E})=-2(n-1)$, or $\mathrm{div}(\overline{E})=2$ (resp.\ $\mathrm{div}(\overline{E})=1$) and $q_X(\overline{E})=-2$. 

%We give a proof for the cases $\mathrm{div}(E)=2(n-1)$, $\mathrm{div}(E)=n-1$, as the proof for the cases $\mathrm{div}(E)=2$, $\mathrm{div}(E)=1$ is respectively the same to that of the cases $\mathrm{div}(E)=2(n-1)$, $\mathrm{div}(E)=n-1$.
\vspace{0.2cm}

\underline{Suppose $\mathrm{div}(\overline{E})=2(n-1)$}.
We observe that the class $\alpha$, which was defined in (\ref{alpha}), has $\mathrm{div}(\alpha)=2(n-1)$, because $\mathrm{div}(\alpha) \geq 2(n-1)$, and $q_X(\alpha)=-2(n-1)$. In particular, the class $\alpha$ is primitive, because if $\alpha=k\alpha'$ for some positive integer $k$, we have $q_X(\alpha)=q_X(\alpha,k\alpha')=2kl(n-1)$, for some integer $l$. This implies $k=1$, as $q_X(\alpha)=-2(n-1)$.
 In this case, we have 
\[
\alpha=2(n-1)y'[D]-[x'-2(n-1)by'][\overline{E}].
\] %(cf. \cite[Section 9]{Mark}). 
Now, we would like to find a solution $(x',y')$ of equation (\ref{pelleqn}) satisfying $x'-2(n-1)by' \equiv 1 \text{ mod } 2(n-1)$. We observe that $x'-2(n-1)by' \equiv x' \text{ mod } 2(n-1)$, for any solution $(x',y')$, hence it suffices to find a solution such that $x' \equiv 1 \text{ mod } 2(n-1)$. Let $(x_1,y_1)$ be the fundamental solution of equation (\ref{pelleqn}). The "second" solution of equation is $x_2=x_1^2+Ny_1^2$, $y_2=2x_1y_1$, where in this case $N=4(n-1)^2b^2-2(n-1)d$. By (\ref{pelleqn}) we have $x_1^2 \equiv 1 \text{ mod } 2(n-1)$, hence $x_2 \equiv 1 \text{ mod } 2(n-1)$, and $(x_2,y_2)$ is a solution we were looking for. 
%We claim that, with respect to this choice of the solution of the Pell equation, the class $\alpha$ is primitive and $\mathrm{div}(\alpha)=2(n-1)$. Indeed, $[D]$ and $[E]$ are linearly independent in $H^2(X,\mathbf{R})$.We observe that $\mathrm{g.c.d.}(2(n-1),x_2-2(n-1)by_2)=1$, because $x_2-2(n-1)by_2 \equiv 1 \text{ mod } 2(n-1)$, and $\mathrm{g.c.d.}(x_2,y_2)=1$, as they satisfy $x_2^2-Ny_2^2=1$. Furthermore
 %\[
 %2(n-1)b\cdot y_2 +[x_2-2(n-1)by_2]=x_2,
 %\]
 %which implies that $\mathrm{g.c.d.}(y_2,x_2-2(n-1)by_2)$ divides $x_2$. Then $\mathrm{g.c.d.}(y_2,x_2-2(n-1)by_2)=1$, because $\mathrm{g.c.d.}(x_2,y_2)=1$, and this implies the primitivity of $\alpha$. 
\vspace{0.2cm}

We now prove the effectivity of $\alpha$. Recall that $q_X(\alpha)=te[(x')^2-N(y')^2]$ (see equality (\ref{squarealpha})). To do so, first observe that  $\left[\frac{\alpha}{2(n-1)}\right]=\left[-\frac{\overline{E}}{2(n-1)}\right]$ in $A_X\cong \mathbf{Z}/2(n-1)\mathbf{Z}$, hence, by using Lemma \ref{Eichlem} and by changing the sign of $-[\overline{E}]$ with the reflection $R_{\overline{E}}\in \mathrm{Mon}^2(X)$, we obtain an isometry $\iota\in \text{O}^+(H^2(X,\mathbf{Z}))$, sending $[\overline{E}]$ to $\alpha$, acting as $-1$ on $A_X$. Indeed, $R_{\overline{E}}$ acts as $-1$ on $A_X$, and the isometry given by Lemma \ref{Eichlem} acts trivially on $A_X$. By \cite[Lemma 4.2]{Mark4} (but see also \cite[Lemma 9.2]{Mark}), the isometry $\iota$ belongs to $\mathrm{Mon}^2(X)$.
%, and by \cite[Proposition 9.16, item $(3)$]{Mark2}, the monodromy orbit of $\alpha$ is determined by $\mathrm{div}(\alpha)$ and $rs(\alpha)$, hence $rs(\alpha)=rs([E])$ (see \cite{Mark2} for the definition of the latter invariant). 
Then $\alpha$ lies in the monodromy orbit of $[\overline{E}]$. Thus, by item \textit{(b)} of Lemma \ref{tecnolemma}, 
%and Markman's characterization of stably exceptional classes on K3$^{[n]}$-type IHS manifolds (cf. \cite[Theorem 1.8]{Mark2})
the class $\alpha$ is effective. 
\vspace{0.2cm}

\underline{Suppose $\mathrm{div}(\overline{E})=n-1$}. 
 We would like to find a solution $(x',y')$ of equation (\ref{pelleqn}) satisfying $x'-(n-1)by' \equiv 1 \text{ mod } 2(n-1)$. Also in this case the "second" solution $(x_2,y_2)$ of equation (\ref{pelleqn}) yields the conclusion. Indeed $N=(n-1)^2b^2-2(n-1)d$ in this case, and we have
\[
x_2 = x_1^2+Ny_1^2= 1+2Ny_1^2 \equiv 1 \text{ mod } 2(n-1).
\]
But $y_2$ is even, hence $x_2-(n-1)by_2 \equiv 1 \text{ mod } 2(n-1)$. Now, we observe that with respect to the solution $(x_2,y_2)$ of (\ref{pelleqn}), we have $\mathrm{div}(\alpha)=n-1$. Indeed, $\mathrm{div}(\alpha) \geq n-1$, and we can have either $\mathrm{div}(\alpha)=n-1$, or $\mathrm{div}(\alpha)= 2(n-1)$, because $q_X(\alpha)=2(n-1)$. If $N$ is even, we have that $x'$ is odd and $y'$ is even, for any solution $(x',y')$. If $N$ is odd, then we can have two possibilities: $x_1$ is odd and $y_1$ is even, or $x_1$ is even and $y_1$ is odd. In any case, we will have that $x_2$ is odd and $y_2$ is even. By definition of divisibility, there exists an element $\gamma \in H^2(X,\mathbf{Z})$ such that $q_X(\overline{E},\gamma)=n-1$, thus $q_X(\alpha, \gamma) \equiv -x_2(n-1) \text{ mod } 2(n-1)$. As $x_2$ is odd, $q_X(\alpha, \gamma)$ cannot be divided by $2(n-1)$, hence $\mathrm{div}(\alpha)=n-1$. Note that with respect to the solution $(x_2,y_2)$, the class $\alpha$ is primitive. Indeed $\mathrm{div}(\alpha)=n-1$, hence either $\alpha$ or $\alpha/2$ is primitive. But $x_2$ is odd and $[\overline{E}]$ is a primitive class, hence $\alpha$ is primitive. To show that $\alpha$ is effective in this case, we note that $[\alpha/(n-1)]=[-\overline{E}/(n-1)]$ in $A_X$, and, arguing as in the case of divisibility $2(n-1)$, we obtain an isometry $\iota \in \text{O}^+(H^2(X,\mathbf{Z}))$, sending $[\overline{E}]$ to $\alpha$, and acting as $-1$ on $A_X$. This isometry is a monodromy operator, by \cite[Lemma 4.2]{Mark4}. 
%By \cite[Proposition 9.16, item $(3)$]{Mark2}, the monodromy orbit of $\alpha$ is determined by  $\mathrm{div}(\alpha)$ and $rs(\alpha)$, hence $rs(\alpha)=rs([E])$. 
Again, since $\alpha$ lies in the monodromy orbit of $[\overline{E}]$, by item \textit{(b)} of Lemma \ref{tecnolemma} 
%and \cite[Theorem 1.11]{Mark2}, 
we conclude that $\alpha$ is effective.
\vspace{0.2cm}

\underline{Suppose $\mathrm{div}(\overline{E})=2$}. We have $\mathrm{div}(\alpha)=2$, because $\mathrm{div}(\alpha) \geq 2$, and $q_X(\alpha)=-2$. Moreover, the class $\alpha$ is primitive, since $q_X(\alpha)=-2$. By \cite[Proposition 1.8]{Mark2}, $\alpha$ and $[\overline{E}]$ lie in the same monodromy orbit, hence $\alpha$ is stably exceptional, whence effective.
\vspace{0.2cm}

\underline{Suppose $\mathrm{div}(\overline{E})=1$}. As above, the class $\alpha$ is primitive. Now we look for a solution $(x',y')$ of equation (\ref{pelleqn}) such that $\mathrm{div}(\alpha)=1$. Also in this case the "second" solution $(x_2,y_2)$ of equation (\ref{pelleqn}) gives what is wanted. Indeed, arguing as in the case $\mathrm{div}(\overline{E})=n-1$, we have that $x_2$ is odd and $y_2$ is even. By definition of divisibility, there exists an element $\gamma \in H^2(X,\mathbf{Z})$ such that $q_X(\overline{E},\gamma)=1$, thus $q_X(\alpha, \gamma) \equiv -x_2 \text{ mod } 2$. As $x_2$ is odd, $q_X(\alpha, \gamma)$ cannot be divided by $2$. Since either $\mathrm{div}(\alpha)=1$, or $\mathrm{div}(\alpha)=2$, we must have $\mathrm{div}(\alpha)=1$. Then, with respect to this choice, by \cite[Proposition 1.8]{Mark2}, $\alpha$ and $[\overline{E}]$ lie in the same monodromy orbit, hence $\alpha$ is stably exceptional, whence effective.
\vspace{0.2cm}

$\bullet$ Suppose now that $X$ is of $\mathrm{Kum}_n$-type. By \cite[Proposition 5.4]{Yosh} we can have $\mathrm{div}(\overline{E})=2(n+1)$ or $\mathrm{div}(\overline{E})=n+1$, and in both cases the BBF square of $\overline{E}$ is $-2(n+1)$. Recall that $A_X \cong \mathbf{Z}/2(n+1)\mathbf{Z}$ %(cf. \cite{Yosh}) 
in this case. The proof goes as in the K3$^{[n]}$-type case.  Also in this case the "second" solution $(x_2,y_2)$ of equation (\ref{pelleqn}) satisfies the congruence $x_2-\mathrm{div}(\overline{E})by_2 \equiv 1 \text{ mod } 2(n+1)$, hence the classes 
\[
\alpha=2(n+1)y_2[D]-[x_2-\mathrm{div}(\overline{E})by_2][\overline{E}]
\] 
and $[\overline{E}]$ are such that $\left[\frac{\alpha}{\mathrm{div}(\overline{E})}\right]=\left[-\frac{\overline{E}}{\mathrm{div}(\overline{E})}\right]$ in $A_X$. Arguing as in the K3$^{[n]}$-type case, we conclude that $\mathrm{div}(\alpha)=\mathrm{div}(\overline{E})$ \footnote{Note that if $\mathrm{div}(\overline{E})=2(n+1)$, then $\mathrm{div}(\alpha)=\mathrm{div}(\overline{E})$ for any solution of (\ref{pelleqn}). Whereas, if $\mathrm{div}(\overline{E})=n+1$, then $\mathrm{div}(\alpha)=\mathrm{div}(\overline{E})$ when we choose the second solution $(x_2,y_2)$ of (\ref{pelleqn}).}, and $\alpha$ is primitive (with respect to the solution $(x_2,y_2)$). Again, arguing exactly as in the K3$^{[n]}$-type case, we obtain an isometry $\iota \in{O}^{+}(H^2(X,\mathbf{Z}))$ of determinant $-1$, acting as $-1$ on $A_X$, and sending $[\overline{E}]$ to $\alpha$. By Markman's and Mongardi's characterization of $\mathrm{Mon}^2(X)$ (cf. \cite[Theorem 1.4]{Mark3} and \cite[Theorem 2.3]{Mon2}), we conclude that the isometry $\iota$ lies in $\mathrm{Mon}^2(X)$. %By Proposition \ref{monodinvar}, the monodromy orbit of $\alpha$ is determined by  $\mathrm{div}(\alpha)$ and $rs(\alpha)$, hence $rs(\alpha)=rs([E])$. 
Then the class $\alpha$ lies in the monodromy orbit of $[\overline{E}]$, hence is effective by item \textit{(b)} of Lemma \ref{tecnolemma}.

%and Yoshioka's characterization of the stably exceptional classes on $\mathrm{Kum}_n$-type IHS manifolds (cf. \cite[Proposition 5.4]{Yosh}).
\end{proof5}

We now show how the computations of this section allow us to prove Theorem \ref{thmden} for the known deformation classes of IHS manifolds, without the use of hyperbolic geometry. The proof follows Kovács' and Huybrechts' strategies (see \cite[Section 8]{Huy2} for the proof of Kovács' result by Huybrechts).

\begin{alternativeproof}
We show that $\EffX$ does not contain circular parts when $\mathrm{Neg}(X)\neq \emptyset$. The proof that
\[
\overline{\mathrm{Eff}(X)}=\overline{\sum_E\mathbf{R}^{\geq 0}[E]}
\]
is the same as that contained in the proof of Theorem \ref{thmden}. Our strategy is as follows. Let $E$ be a prime exceptional divisor and, as above, let $\overline{E}$ be an integral divisor whose class is primitive and such that $[E]$ is a positive multiple of $[\overline{E}]$. Firstly, assuming that $\EffX$ contains a circular part $\mathcal{C}$, we show that there exists a rational isotropic ray $R$ in $\mathcal{C}$. Secondly, using the class $[\overline{E}]$ and an integral generator $\alpha$ of $R$, we construct effective classes $\alpha_k$ with $q_X(\alpha_k)<0$, such that the sequence of rays $\{\mathbf{R}^{> 0}\alpha_k\}_k$ converges to $R=\mathbf{R}^{> 0}\alpha$, for $k\to +\infty$. Since the BBF square of the $\alpha_k$ is negative, we will reach a contradiction. 
\vspace{0.2cm}

Set $q_X(\overline{E})=te$, where $\mathrm{div}(\overline{E})=t$, and $e$ is a negative integer. Assume by contradiction that $\EffX$ has a circular part $\mathcal{C}$. We can assume that $\mathcal{C}=\mathbf{R}^{> 0}\mathcal{C}$. Moreover, as we have already explained in the proof of Theorem \ref{thmden}, we have $\mathcal{C} \subset \partial \overline{\mathscr{C}_X} \cap \partial \EffX$. Then there exists a neighborhood $U$ of $\mathcal{C}$ in $\EffX$, such that $q_X(\beta)\geq 0$, for any $\beta \in U$.
\vspace{0.2cm}

Let $D$ be any integral divisor with $[D] \in \mathscr{C}_X$. By Proposition \ref{thmden2}, on the opposite side of $[D]$ with respect to $[\overline{E}]$, there exists an effective integral class $\alpha$ that is of BBF square $0$, or of BBF square $te$. When the class $[D]$ approaches $\mathcal{C}$, we can only have $q_X(\alpha)=0$. In particular, the circular part $\mathcal{C}$ contains an effective integral isotropic class $\alpha$. Now, consider the class $\alpha_{\epsilon}=(1-\epsilon)\alpha+\epsilon [\overline{E}]$, where $\epsilon > 0$ is rational. Clearly, for $\epsilon$ small enough, the class $\alpha_{\epsilon}$ is effective. Thus we have $q_X(\alpha,\overline{E})>0$, as otherwise we would have \[
q_X(\alpha_{\epsilon})=-2\epsilon^2q_X(\alpha,\overline{E})+2\epsilon q_X(\alpha,\overline{E})+\epsilon^2te<0,\]
 contradicting $\alpha \in \partial{\EffX} \cap \partial \overline{\mathscr{C}_X}$. 
\vspace{0.2cm}
 
 Since by assumption $\rho(X)\geq 3$, we can find an integral class $\alpha'$ of negative BBF square belonging to $\left(\mathbf{R} \alpha \oplus \mathbf{R} [\overline{E}]\right)^{\perp}$. We now define
\begin{equation}\label{eqn1}
\alpha_k:=-2k^2q_X(\alpha')\left(q_X(\alpha,\overline{E})\right)^3\alpha-2k\left(q_X(\alpha,\overline{E})\right)^2\alpha'+[\overline{E}].
\end{equation}
An easy computation shows that $q_X(\alpha_k)=q_X(\overline{E})=te$ and that, for $k \gg 0$, we have $q_X(\alpha_k,[A])>0$, where $A$ is any ample divisor. We want to show that for $k$ large enough the classes $\alpha_k$ are effective. First of all, for all known deformation classes of IHS manifolds, we observe that $\mathrm{div}(\alpha_k)=\mathrm{div}(\overline{E})$. Indeed, by construction $\mathrm{div}(\alpha_k)\geq \mathrm{div}(\overline{E})$. Moreover, by the explicit description of the primitive stably exceptional classes for the known deformation classes of IHS manifolds, since $q_X(\alpha_k)=q_X(\overline{E})$, we have either $q_X(\alpha_k)=-\mathrm{div}(\overline{E})$, or $q_X(\alpha_k)=-2 \cdot \mathrm{div}(\overline{E})$. In particular, we have either $\mathrm{div}(\alpha_k)=\mathrm{div}(\overline{E})$, or $\mathrm{div}(\alpha_k)=2\cdot \mathrm{div}(\overline{E})$. To exclude the second case, pick an element $\gamma$ such that $q_X(\overline{E},\gamma)=\mathrm{div}(\overline{E})$. Then $q_X(\alpha_k,\gamma)$ is divisible by $\mathrm{div}(\overline{E})$, but not by $2\cdot \mathrm{div}(\overline{E})$, hence $\mathrm{div}(\alpha_k)=\mathrm{div}(\overline{E})$.
Also, we observe that the classes $\alpha_k$ are primitive, independently of the deformation type we have chosen. Indeed, $[\overline{E}], \alpha$ and $\alpha'$ are linearly independent in $N^1(X)_{\mathbf{R}}$, and the class $[\overline{E}]$ is primitive. If $q_X(\alpha_k)=-\mathrm{div}(\overline{E})$, we are done, otherwise $q_X(\alpha_k)=-2\cdot \mathrm{div}(\overline{E})$, and either $\alpha_k$ or $\alpha_k/2$ is primitive. The primitivity of $[\overline{E}]$ forces $\alpha_k$ to be primitive.
We can now conclude that for $k$ large enough the classes $\alpha_k$ are effective. To do so, we distinguish the different deformation types.
\vspace{0.2cm}

$\bullet$ Suppose that $X$ is of OG10-type. By \cite[Proposition 3.1]{Mon} we have either $q_X(\overline{E})=-2$ and $\mathrm{div}(\overline{E})=1$, or $q_X(\overline{E})=-6$ and $\mathrm{div}(\overline{E})=3$. By \cite[Proposition 3.1]{Mon}, for $k$ large enough, the classes $\alpha_k$ are stably exceptional, hence effective. 
\vspace{0.2cm}

$\bullet$  Suppose that $X$ is of OG6-type. By \cite[Proposition 6.8]{Mon1} we have either $q_X(\overline{E})=-4$ and $\mathrm{div}(\overline{E})=2$, or $q_X(\overline{E})=-2$ and $\mathrm{div}(\overline{E})=2$. By \cite[Proposition 6.8]{Mon1}, if $k$ is large enough, the classes $\alpha_k$ are stably exceptional, hence effective.
\vspace{0.2cm}

$\bullet$ Suppose that $X$ is of K3$^{[n]}$-type. By \cite[Theorem 1.11]{Mark2}, we have $\mathrm{div}(\overline{E})=2(n-1)$ (resp.\ $\mathrm{div}(\overline{E})=n-1$), and $q_X(\overline{E})=-2(n-1)$, or $\mathrm{div}(\overline{E})=2$ (resp.\ $\mathrm{div}(\overline{E})=1$) and $q_X(\overline{E})=-2$. If $q_X(\overline{E})=-2(n-1)$, we observe that $[\alpha_k/\mathrm{div}(\alpha_k)]=[\overline{E}/\mathrm{div}(\overline{E})]$ in $A_X$. Hence, by Lemma \ref{Eichlem}, we obtain an isometry $\iota \in \widetilde{\text{SO}}^+(H^2(X,\mathbf{Z}))$ sending $[\overline{E}]$ to $\alpha_k$. In particular, $\iota\in \mathrm{Mon}^2(X)$, by \cite[Lemma 4.2]{Mark4}. Thus we conclude that the classes $\alpha_k$ are stably exceptional and hence effective for $k$ large enough. If $q_X(\overline{E})=-2$, we conclude by \cite[Proposition 1.8]{Mark2} that the classes $\alpha_k$ are stably exceptional and hence effective, for $k$ large enough.
\vspace{0.2cm}

$\bullet$ Suppose that $X$ is of $\mathrm{Kum}_n$-type. By \cite[Proposition 5.4]{Yosh}, we have $\mathrm{div}(\overline{E})=2(n+1)$, or $\mathrm{div}(\overline{E})=n+1$, and in both cases the BBF square of $\overline{E}$ is $-2(n+1)$. We observe that $[\alpha_k/\mathrm{div}(\alpha_k)]=[\overline{E}/\mathrm{div}(\overline{E})]$ in $A_X$. By Lemma \ref{Eichlem}, we obtain an isometry $\iota \in \widetilde{\text{SO}}^+(H^2(X,\mathbf{Z}))$ sending $[\overline{E}]$ to $\alpha_k$. By Markman's and Mongardi's characterization of $\mathrm{Mon}^2(X)$ (cf. \cite[Theorem 1.4]{Mark3} and \cite[Theorem 2.3]{Mon2}), we conclude that the isometry $\iota$ lies in $\mathrm{Mon}^2(X)$. Then the classes $\alpha_k$ are stably exceptional and hence effective, for $k$ large enough.
\vspace{0.2cm}

Now, dividing both members of (\ref{eqn1}) by $2k^2$ we see that the sequence of rays $\left\{\mathbf{R}^{>0}\alpha_k\right\}_k$ converges to $\mathbf{R}^{>0}\alpha$ and this contradicts the circularity of $\EffX$ at $\gamma$ (in particular the fact that locally around $\gamma$ in $\overline{\mathrm{Eff}(X)}$ the BBF form is nonnegative).
\end{alternativeproof}

\begin{rmk}\label{effectivedivisors}
We observe that if in Proposition \ref{thmden2} we are in the first situation, using the classes $\alpha_k$ defined in equation (\ref{eqn1}), we can construct explicit effective divisors with the same monodromy orbit as that of a given prime exceptional divisor $E$. 
\end{rmk}

\printbibliography

\end{document}